\documentclass[12pt]{article}

\usepackage[margin=1in]{geometry} 
\usepackage{amsmath,amsthm,amssymb,color}
\usepackage{graphicx}
\usepackage{mathtools}
\usepackage[section]{placeins}
\usepackage{listings}
\usepackage{mathdots}
\usepackage{centernot}
\usepackage{algorithm}
\usepackage[noend]{algpseudocode}
\usepackage{csquotes}
\usepackage[table,xcdraw]{xcolor}
\usepackage{booktabs}
\usepackage{enumerate}
\usepackage{tikz}
\usetikzlibrary{calc}
\usepackage{bbm}
\usetikzlibrary{arrows}
\usetikzlibrary{patterns}

\newcommand\independent{\protect\mathpalette{\protect\independenT}{\perp}}
\def\independenT#1#2{\mathrel{\rlap{$#1#2$}\mkern2mu{#1#2}}}

\theoremstyle{plain}
\newtheorem{thm}{Theorem}
\theoremstyle{plain}
\newtheorem{lem}{Lemma}
\theoremstyle{plain}
\theoremstyle{plain}
\newtheorem{cor}{Corollary}
\theoremstyle{plain}
\newtheorem{defin}{Definition}
\theoremstyle{remark}
\theoremstyle{plain}
\theoremstyle{plain}
\theoremstyle{plain}
\theoremstyle{plain}

\newcommand{\mR}{\mathbb{R}}

\newcommand{\E}{\mathbb{E}}

\makeatletter
\def\BState{\State\hskip-\ALG@thistlm}
\makeatother

\begin{document} 	

 
\title{\LARGE \bf Causal Models on Probability Spaces} 
\author{\Large Irineo Cabreros\thanks{Program in Applied and Computational Mathematics, Princeton University, Princeton, NJ 08544 USA. Email: \texttt{cabreros@math.princeton.edu}.}  \ and John D. Storey\thanks{Lewis-Sigler Institute for Integrative Genomics, Princeton University, Princeton, NJ 08544 USA. Email: \texttt{jstorey@princeton.edu}.}}
\date{}
\maketitle

\begin{abstract}
We describe the interface between measure theoretic probability and causal inference by constructing causal models on probability spaces within the potential outcomes framework. We find that measure theory provides a precise and instructive language for causality and that consideration of the probability spaces underlying causal models offers clarity into central concepts of causal inference. By closely studying simple, instructive examples, we demonstrate insights into causal effects, causal interactions, matching procedures, and randomization. Additionally, we introduce a simple technique for visualizing causal models on probability spaces that is useful both for generating examples and developing causal intuition. Finally, we provide an axiomatic framework for causality and make initial steps towards a formal theory of general causal models. 
\end{abstract}

\section{Introduction}

The goal of causal inference is to understand mechanistic relationships between random variables. Beyond simply observing that smokers have a higher rate of lung cancer than non-smokers,  for instance, causal inference aims to determine whether lung cancer is a downstream effect of the act of smoking. As random variables are probabilistic objects, probability theory is intrinsic to causality. 



Despite the centrality of probability in causal inference, the precise relationship between the two has historically been contested. For instance, it has long been emphasized that probabilistic relationships often have no causal interpretation, as any discussion of causality is quick to remark that ``correlation is not causation." The earliest recorded distinctions between dependence and causation predate the introduction of the correlation coefficient itself. Fechner, who in 1851 differentiated between a ``causal dependency" and a ``functional relationship" in his work on mathematical psychology \cite{fechner}, is possibly the first to articulate this distinction \cite{heckman}. 

In contrast, Karl Pearson, the eponym of the Pearson correlation, held that correlation subsumed causation. In his influential book \textit{The Grammar of Science} \cite{pearson_1910}, Pearson states:
\begin{quote}
It is this conception of correlation between two occurrences embracing all relationships from absolute independence to complete dependence, which is the wider category by which we have to replace the old idea of causation. 
\end{quote}

\noindent To Pearson, causation was simply perfect co-occurance: a correlation coefficient of exactly $\pm 1$ \cite{aldrich1995}. Notions of causality beyond probabilistic correlation, Pearson argued, were outside the realm of scientific inquiry \cite{pearl_why}.

Pearson's view of causality is far from the main formulations of causality today. Under our modern understanding of causality, one can easily construct examples in which $X$ and $Y$ have correlation 1, however neither $X$ is causal for $Y$ nor $Y$ is causal for $X$. Both $X$ and $Y$ may be the result of some common confounding cause, for instance. Likewise, one can construct examples of systems in which the observed correlation between $X$ and $Y$ is 0, however $X$ is causal for $Y$. $X$ may be confounded with a third variable $Z$, which masks the effect of $X$ on $Y$ in the population. Causality and correlation are now viewed as conceptually distinct phenomena. 

The earliest attempts to define causality in a manner that resembles our current conception avoided probabilistic language altogether. A representative example of an early definition of causality, typically credited to Marshall \cite{marshall_1890}  though likely of earlier origins \cite{ceteris}, is paraphrased as follows.

\begin{defin}[Early notion of causality (\textit{Ceteris Paribus})]\label{def:informal_causality}
\label{causal_definition}
$X$ is said to be causal for $Y$ if directly manipulating the value of $X$, keeping everything else unchanged, changes the value of $Y$. 
\end{defin}



\noindent While Definition \ref{def:informal_causality} is intuitively appealing---providing a practical description of causality for controlled laboratory settings---it clearly lacks mathematical rigor. In particular, it is unclear how to translate the idea of a ``direct manipulation" into probabilistic language. 

Viewed within the measure theoretic framework of probability, Definition \ref{def:informal_causality} is particularly problematic. A pair of random variables $X$ and $Y$ defined on the same probability space $(\Omega, \mathcal{F}, P)$ are determined by a common source of randomness: the selection of a random outcome $\omega \in \Omega$. Thus, it is not at all clear why ``directly" manipulating the value of $X$ would have an impact on $Y$. Classical probability allows random variables to convey information about each other, but only through the symmetric notion of probabilistic dependence. Conversely, causal inference hopes to distinguish directionality; the statement ``smoking causes lung cancer" is distinct from the statement ``lung cancer causes smoking." Where causal inference seeks to draw arrows between random variables ($X \rightarrow Y$), classical probability treats $X$ and $Y$ symmetrically in that both are functions of a single random outcome, $X(\omega)$ and $Y(\omega)$.  


The central aim of this work is to clearly explain how causal models can be constructed within the measure theoretic framework of classical probability theory. We take as our starting point the Neyman-Rubin model (NRM) of potential outcomes \cite{neyman, rubin_1974, holland_1986}, and describe the structure of the probability space on which these potential outcomes are defined. From this perspective, we will see that a precise definition of causality can be couched in the standard probabilistic language of measure theory. Rather than defining causality in terms of ``direct manipulations" of $X$, we will define $X$ as causal for $Y$ if the potential outcomes $Y_{X = x}$ are unequal on subsets of nonzero measure. We emphasize throughout this work that causal models are probabilistic models with structured constraints between observed and unobserved (i.e., potential outcome) random variables.

We should be clear that we do not claim to unify probability theory with causality. The notion \textit{ceteris paribus} from Definition \ref{def:informal_causality} was formalized in probabilistic language as early as 1944 by Haavelmo \cite{haavelmo_1944, heckman}. Today, probability is the common language of all modern causal inference frameworks. Within the Directed Acyclic Graphs (DAG) framework, causal relationships are discovered by searching for sets random variables satisfying certain conditional independence relationships \cite{pearl_book, sgs_book, lauritzen_book}. Within the potential outcomes framework of causality \cite{neyman, rubin_1974, holland_1986}, the primary goal is to estimate causal effects, defined in terms of expectations of partially observable random variables (e.g., the $\text{ACE}$). In each framework, causal relationships map onto probabilistic relationships, which are in turn diagnosed by statistical tests. The contribution of the present work is not to unify causality with probability, but rather to explicate fundamental concepts of modern causal inference in the language of measure theory.

Clarifying the interface between causality and measure theory is useful for several reasons. First, measure theory provides a simplifying perspective for understanding the basic framework of causality. Classical probability theory, we will find, is completely sufficient to describe causal models. Second, the measure theoretic perspective is an insightful one. For instance, we find that consideration of the underlying probability spaces provides insight into experimental procedures (such as randomization) and non-experimental procedures (such as matching). Additionally, a simple method of visualizing causal models on probability spaces, which we employ throughout this work, enables one to generate and reason about a rich set of instructive examples. Third, by making explicit the relationship between causality and measure theory, we hope to initiate interest in applying the tools from measure theory to further develop causal inference. 

The remainder of this work is organized as follows. In Section \ref{sec:review} we provide a brief overview of the measure theoretic framework of probability theory. We also introduce examples, notation, and a method for visualizing causal models that will frequently be used in later sections. In Section \ref{sec:two_vars} we closely examine the simplest causal system: two binary random variables. Here we review the potential outcomes framework within the language of probability spaces, emphasizing that potential outcomes are simply random variables in the familiar sense of classical probability theory. We also introduce a  formal definition of causality and a formal model for experimental randomization in this simple system. In Section \ref{sec:three_vars}, we consider a system of three binary random variables, an incrementally more complex system that introduces several new conceptual challenges. First, we see how two random variables may be jointly causal for a third random variable, despite neither being individually causal. We also re-examine the concept of matching--a popular method of causal inference in the observational setting--from the measure theoretic perspective. Finally, in Section \ref{sec:generalization} we expand the ideas developed in Sections \ref{sec:two_vars} and \ref{sec:three_vars} to more general causal models.

\section{Background and notation: probability spaces and visual representation}\label{sec:review}

In the present section, we provide a brief review of the measure theoretic framework of classical probability theory, both to establish notation and to introduce a method for visualizing probabilistic systems that we will use throughout this work. For a more detailed review of classical probability theory, please refer to Appendix \ref{app:detailed_background}

The central construct within the measure theoretic framework of probability theory is the \textbf{\textit{probability space}}. Denoted by the triple $(\Omega, \mathcal{F}, P)$, the probability space consists of a \textbf{\textit{sample space}} ($\Omega$), a \textbf{\textit{$\sigma$-algebra}} ($\mathcal{F}$), and a \textbf{\textit{probability measure}} ($P$). A \textit{\textbf{random variable}} $X$ is an $\mathcal{F}$-measurable function, mapping elements $\omega \in \Omega$ (called \textit{\textbf{random outcomes}}) to $\mathbb{R}$. Somewhat counter-intuitively, random variables are deterministic functions of $\omega$. Perfect knowledge of $\omega$ implies perfect knowledge of a random variable; uncertainty in $\omega$ results in uncertainty in a random variable. A random variable $X$ and a probability measure $P$ together define the \textbf{\textit{probability law}} $P_X$, which maps elements $B$ of the Borel $\sigma$-algebra $\mathcal{B}$ to $\mathbb{R}$ as follows:
\begin{align*}
P_X(B) &\equiv P( X^{-1} (B) )
\end{align*}
\noindent The goal of classical statistical inference is to understand the probability law $P_X$ from observed realizations of the random variable $X(\omega)$.

Throughout this work, we will find it convenient to visually represent random variables on a simple probability space probability space, which we call the \textit{\textbf{square space}}. The square space is defined by the triple $(\Omega, \mathcal{F}, P) = ([0,1]^2, \mathcal{B}_2, \mu_2)$, where the sample space $[0,1]^2$ is the unit square in $\mR^2$, $\mathcal{B}_2$ is the Borel $\sigma$-algebra on $[0,1]^2$, and $\mu_2$ is the two-dimensional Lebesgue measure (equivalent to the common notion of ``area"). We will find the square space particularly useful because it is both amenable to visualization and flexible enough to accommodate many probabilistic systems. In Figure \ref{fig:rv}, we represent a binary random variable $X$ on the square space. In this, and in all following examples, shaded regions of the sample space correspond to the pre-image of 1 for the corresponding binary random variable. Therefore, all points in the upper half of $\Omega$ map to 1 and all points in the lower half of $\Omega$ map to 0. Since the underlying probability measure is the Lebesgue measure, the probability law for $X$ is that of a fair coin: $P_X(0) = P_X(1) = \frac{1}{2}$.

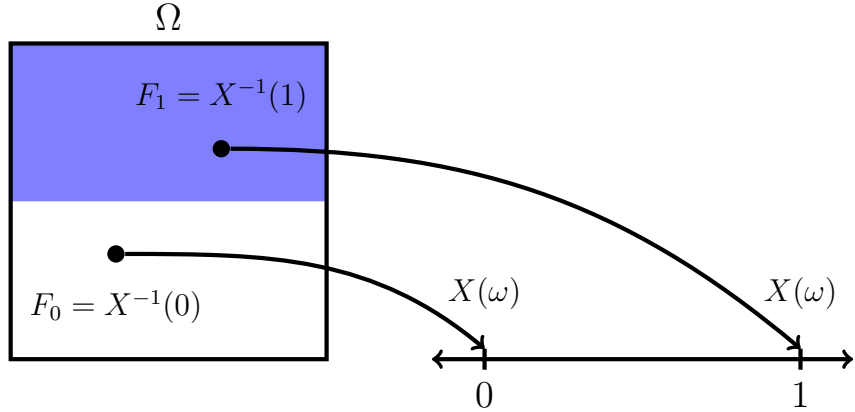
\begin{figure}
\begin{center}
\begin{tikzpicture}[scale = 0.7]

\coordinate (CO) at (-3, 3);
\coordinate (CA) at (6, 3);
\coordinate (F1) at ($(CO) + (0, -2.2)$);
\coordinate (F0) at ($(CO) + (2.5, -2.2)$);

\coordinate (nw) at (-3, 3);
\coordinate (ne) at (3, 3);
\coordinate (se) at (3, -3);
\coordinate (sw) at (-3, -3);

\coordinate (nwa) at (-3, 3);
\coordinate (nea) at (3, 3);
\coordinate (sea) at (3, -3);
\coordinate (swa) at (-3, -3);

\coordinate (omega_dot) at ( $(CO) + (1, -1)$ );
\coordinate (r2_dot) at ( $(CA) + (2, 2)$ );
\coordinate (y_dot) at ( $(CA) + (-2, 2)$ );
\coordinate (x_dot) at ( $(CA) + (2, -2)$ );

\fill[blue, opacity = 0.5] ($(CO) + (ne)$) -- ($(CO) + (3, 0)$) -- ($(CO) + (-3, 0)$) -- ($(CO) + (nw)$) -- cycle;

\draw[ultra thick] ( $(CO) + (nw)$ ) -- ( $(CO) + (ne)$ ) -- ( $(CO) + (se)$ ) -- ( $(CO) +  (sw)$ ) -- cycle;
\draw ( $(CO) + (0, 3.5)$ ) node {\large $\Omega$};

\draw[<->, ultra thick] ($(CA) + (-4, -3) $) -- ( $(CA) + (4, -3)$);
\draw[-, ultra thick] ( $(CA) + (-3, -3 + 0.2)$ ) -- ( $(CA) + (-3, -3 - 0.2)$ ) node[below] {\large $0$};
\draw[-, ultra thick] ( $(CA) + (3, -3 + 0.2)$ ) -- ( $(CA) + (3, -3 - 0.2)$ ) node[below] {\large $1$};

\tikzstyle{vertex_unobserved}=[circle,fill=black, text width = 1mm, align = center, inner sep=1.5pt]
\path
($(CO) + (1, 1)$) node[vertex_unobserved](omega1) {}
($(CO) + (-1, -1)$) node[vertex_unobserved](omega2) {}  
($(CO) + (-1, -2)$) node(F1) {$F_0 = X^{-1}(0)$}
($(CO) + (1, 2)$) node(F2) {$F_1 = X^{-1}(1)$}
($(CO) + (12, -1.75)$) node(arrow1) {$X(\omega)$}
($(CO) + (6, -1.75)$) node(arrow1) {$X(\omega)$};
\draw[->, ultra thick] (omega2) to [out=0,in= 140] ($(CA) + (-3, -3 + 0.2)$);
\draw[->, ultra thick] (omega1) to [out=0,in=140] ($(CA) + (3, -3 + 0.2)$);

\end{tikzpicture}
\end{center}
\caption{A binary random variable $X$ on the square sample. Shaded regions denote the pre-image of 1. Black points correspond to individual elements of the sample space. }\label{fig:rv}
\end{figure}

Multiple random variables can be defined on a single probability space with multivariate probability laws defined in the natural way. If $X$ and $Y$ are two random variables defined on $(\Omega, \mathcal{F}, P)$, then the multivariate random variable $(X, Y)$ is defined as the following map between $\Omega$ and $\mathbb{R}^2$:
\begin{align*}
(X,Y)(\omega) = (X(\omega), Y(\omega)) \in \mR^2
\end{align*}
\noindent The joint probability law $P_{X,Y}$ is defined as a map between $\mathcal{B}_2$, the Borel $\sigma$-algebra on $\mathbb{R}^2$, and $\mathbb{R}$:
\begin{align*}
P_{X,Y}(B_2) = P( (X,Y)^{-1} (B_2) )
\end{align*}
\noindent If for any Borel rectangle $B_X \times B_Y \in \mathcal{B}_2$, we have the relationship
\begin{align*}
P_{X, Y}(B_X \times B_Y) = P_X(B_X)P_Y(B_Y)
\end{align*}
\noindent then $X$ and $Y$ are called \textbf{\textit{independent}}. Otherwise, $X$ and $Y$ are \textbf{\textit{dependent}}. 

In Figure \ref{fig:mult}, we represent two binary random variables $X$ and $Y$ simultaneously on the square space. $X$ is defined as in Figure \ref{fig:rv}, while $Y$ maps $\omega$ from the upper right triangle to 1. The region where both $X$ and $Y$ map $\omega$ to 1 is shaded darker; in this region, $(X, Y)(\omega) = (1, 1)$. In this example, $X$ and $Y$ are dependent. This can be seen qualitatively from Figure \ref{fig:mult} by noting that the distribution of $Y$ differs on the subsets $X^{-1}(1)$ and $X^{-1}(0)$.

\begin{figure}
\begin{center}
\begin{tikzpicture}[scale = 0.7]

\coordinate (CO) at (-3, 3);
\coordinate (CA) at (5, 3);

\coordinate (nw) at (-3, 3);
\coordinate (ne) at (3, 3);
\coordinate (se) at (3, -3);
\coordinate (sw) at (-3, -3);

\coordinate (nwa) at (-3, 3);
\coordinate (nea) at (3, 3);
\coordinate (sea) at (3, -3);
\coordinate (swa) at (-3, -3);

\coordinate (omega1) at ( $(CO) + (-1.5, -2.5)$ );
\coordinate (omega2) at ( $(CO) + (1.5, -2.5)$ );
\coordinate (omega3) at ( $(CO) + (1.5, 2.5)$ );
\coordinate (omega4) at ( $(CO) + (-1, 2.0)$ );
\coordinate (rdot1) at ( $(CA) + (-2, 2)$ );
\coordinate (rdot2) at ( $(CA) + (-2, -2)$ );
\coordinate (rdot3) at ( $(CA) + (2, -2)$ );
\coordinate (rdot4) at ( $(CA) + (2, 2)$ );


\fill[blue, opacity = 0.5] ($(CO) + (nw)$) -- ($(CO) + (ne)$) -- ($(CO) + (3, 0)$) -- ($(CO) + (-3, 0)$) -- cycle;
\fill[red, opacity = 0.5] ($(CO) + (ne)$) -- ($(CO) + (se)$) -- ($(CO) + (nw)$) -- cycle;

\draw[ultra thick] ( $(CO) + (nw)$ ) -- ( $(CO) + (ne)$ ) -- ( $(CO) + (se)$ ) -- ( $(CO) +  (sw)$ ) -- cycle;
\draw ( $(CO) + (0, 3.5)$ ) node {\large $\Omega$};

\draw[->,ultra thick] ($(CA) + (-2.3, -2) $) -- ( $(CA) + (3, -2)$ ) node[right]{$x$};
\draw[->,ultra thick] ($(CA) + (-2, -2.3) $) -- ( $(CA) + (-2, 3)$ ) node[right]{$y$};
\draw ( $(CA) + (0, 3.5)$ ) node {\large $\mathbb{R}^2$};
\draw[-, ultra thick] ( $(CA) + (-2, -2.3)$ ) -- ( $(CA) + (-2, -2.3)$ ) node[below] {\large $0$};
\draw[-, ultra thick] ( $(CA) + (2, -1.8)$ ) -- ( $(CA) + (2, -2.2)$ ) node[below] {\large $1$};
\draw[-, ultra thick] ( $(CA) + (-2.3, -2)$ ) -- ( $(CA) + (-2.3, -2)$ ) node[left] {\large $0$};
\draw[-, ultra thick] ( $(CA) + (-1.8, 2)$ ) -- ( $(CA) + (-2.2, 2)$ ) node[left] {\large $1$};

\tikzstyle{vertex_observed}=[circle,fill=red, text width = 1mm, align = center, inner sep=1.5pt]
\tikzstyle{vertex_unobserved}=[circle,fill=black, text width = 1mm, align = center, inner sep=1.5pt]
\path (rdot4) node[vertex_unobserved](rdot4) {}
      (omega4) node[vertex_unobserved](omega4) {}
      ($(1.5,3.5)$) node(xy) {$(X, Y)(\omega)$};
\draw[->, ultra thick] (omega4) to [out=340,in=200] (rdot4);
\end{tikzpicture}
\end{center}
\caption{A system of two binary random variables $X$ and $Y$ defined on the square space. The pre-image of 1 for the random variable $X$ is the upper half of $\Omega$. The pre-image of 1 for the random variable $Y$ is the upper right triangle.}\label{fig:mult}
\end{figure}
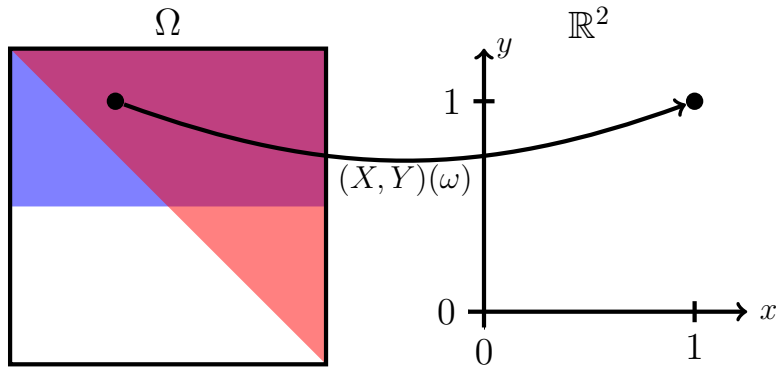

Two probability spaces $(\Omega_1, \mathcal{F}_1, P_1)$ and $(\Omega_2, \mathcal{F}_2, P_2)$ can be used to construct a third probability space, called the \textit{\textbf{product space}}:
\begin{align*}
(\Omega, \mathcal{F}, P) = (\Omega_1 \times \Omega_2, \mathcal{F}_1 \times \mathcal{F}_2, P_1 \times P_2)
\end{align*}
\noindent A feature of the product space construction, which we will make use of in our discussion of experimental randomization, is that it induces independence between random variables. In particular, when $X$ is defined on $(\Omega_1, \mathcal{F}_1, P_1)$ and $Y$ is defined on $(\Omega_2, \mathcal{F}_2, P_2)$, $X$ and $Y$ are independent random variables when defined jointly on the product space $(\Omega_1 \times \Omega_2, \mathcal{F}_1 \times \mathcal{F}_2, P_1 \times P_2)$. Figure \ref{fig:prod_space} displays a product space construction. In this example, 
\begin{align*}
(\Omega_1, \mathcal{F}_1, P_1) = (\Omega_2, \mathcal{F}_2, P_2) = ([0,1], \mathcal{B}_1, \mu_1)
\end{align*}
\noindent where $\mathcal{B}_1$ is the Borel $\sigma$-algebra on $[0,1]$ and $\mu_1$ is the one-dimensional Lebesgue measure. The product space is therefore the square space, $([0, 1]^2, \mathcal{B}_2, \mu_2)$, and $X$ and $Y$ are independent random variables by construction.

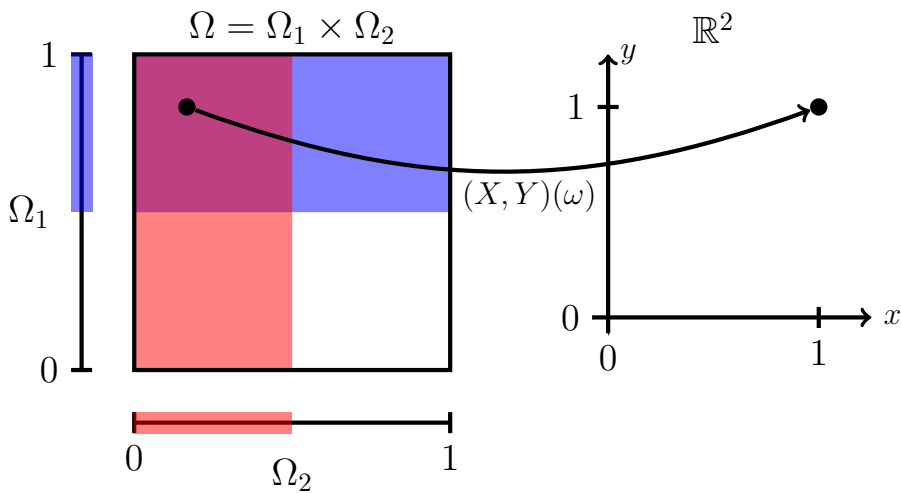
\begin{figure}
\begin{center}
\begin{tikzpicture}[scale = 0.7]

\coordinate (CO) at (5, 3);
\coordinate (CA) at (13, 3);
\coordinate (F) at (-2, 3);

\coordinate (nw) at (-3, 3);
\coordinate (ne) at (3, 3);
\coordinate (se) at (3, -3);
\coordinate (sw) at (-3, -3);

\coordinate (nwa) at (-3, 3);
\coordinate (nea) at (3, 3);
\coordinate (sea) at (3, -3);
\coordinate (swa) at (-3, -3);

\coordinate (omega) at ( $(CO) + (-2, 2)$ );
\coordinate (rdot) at ( $(CA) + (2, 2)$ );


\fill[blue, opacity = 0.5] ($(CO) + (ne)$) -- ($(CO) + (3, 0)$) -- ($(CO) + (-3, 0)$) -- ($(CO) + (nw)$) -- cycle;
\fill[red, opacity = 0.5] ($(CO) + (nw)$) -- ($(CO) + (0, 3)$) -- ($(CO) + (0, -3)$) -- ($(CO) + (sw)$) -- cycle;

\draw[-, ultra thick] ($(CO) + (-3, -4) $) --  ( $(CO) + (3, -4)$);
\draw[-, ultra thick] ( $(CO) + (-3, -4+0.2)$ ) -- ( $(CO) + (-3, -4-0.2)$ ) node[below] 
{\large $0$};
\draw[-, ultra thick] ( $(CO) + (3, -4+0.2)$ ) -- ( $(CO) + (3, -4-0.2)$ ) node[below] 
{\large $1$};
\fill[red, opacity = 0.5] ( $(CO) + (0, -4+0.2)$ ) -- ( $(CO) + (0, -4-0.2)$ ) -- ( $(CO) + (-3, -4-0.2)$ ) -- ( $(CO) + (-3, -4+0.2)$ ) -- cycle;
\path ($(CO) + (0, -5)$) node(O1) {\large $\Omega_2$}; 

\draw[-, ultra thick] ($(CO) + (-4, 3) $) -- ( $(CO) + (-4, -3)$);
\draw[-, ultra thick] ( $(CO) + (-4+0.2, -3)$ ) -- ( $(CO) + (-4-0.2, -3)$ ) node[left] 
{\large $0$};
\draw[-, ultra thick] ( $(CO) + (-4+0.2, 3)$ ) -- ( $(CO) + (-4-0.2, 3)$ ) node[left] 
{\large $1$};
\fill[blue, opacity = 0.5] ( $(CO) + (-4+0.2, 0)$ ) -- ( $(CO) + (-4-0.2, 0)$ ) -- ( $(CO) + (-4-0.2, 3)$ ) -- ( $(CO) + (-4+0.2, 3)$ ) -- cycle;
\path ($(CO) + (-5, 0)$) node(O1) {\large $\Omega_1$};

\draw[ultra thick] ( $(CO) + (nw)$ ) -- ( $(CO) + (ne)$ ) -- ( $(CO) + (se)$ ) -- ( $(CO) +  (sw)$ ) -- cycle;
\draw ( $(CO) + (0, 3.5)$ ) node {\large $\Omega = \Omega_1 \times \Omega_2$};

\draw[->,ultra thick] ($(CA) + (-2.3, -2) $) -- ( $(CA) + (3, -2)$ ) node[right]{$x$};
\draw[->,ultra thick] ($(CA) + (-2, -2.3) $) -- ( $(CA) + (-2, 3)$ ) node[right]{$y$};
\draw ( $(CA) + (0, 3.5)$ ) node {\large $\mathbb{R}^2$};
\draw[-, ultra thick] ( $(CA) + (-2, -2.3)$ ) -- ( $(CA) + (-2, -2.3)$ ) node[below] {\large $0$};
\draw[-, ultra thick] ( $(CA) + (2, -1.8)$ ) -- ( $(CA) + (2, -2.2)$ ) node[below] {\large $1$};
\draw[-, ultra thick] ( $(CA) + (-2.3, -2)$ ) -- ( $(CA) + (-2.3, -2)$ ) node[left] {\large $0$};
\draw[-, ultra thick] ( $(CA) + (-1.8, 2)$ ) -- ( $(CA) + (-2.2, 2)$ ) node[left] {\large $1$};

\tikzstyle{vertex_unobserved}=[circle,fill=black, text width = 1mm, align = center, inner sep=1.5pt]
\path (rdot) node[vertex_unobserved](rdot) {}
      (omega) node[vertex_unobserved](omega) {}
      ($(9.5,3.3)$) node(xy) {$(X, Y)(\omega)$};
\draw[->, ultra thick] (omega) to [out=340,in=200] (rdot);

\end{tikzpicture}
\end{center}
\caption{The probability space $(\Omega, \mathcal{F}, P)$ is formed by the product of spaces the spaces $(\Omega_1, \mathcal{F}_1, P_1)$ and $(\Omega_2, \mathcal{F}_2, P_2)$. The two binary random variables $X$ and $Y$ defined separately on the original probability spaces are independent on the product space.}\label{fig:prod_space}
\end{figure}

Before discussing causal models in the following sections, it is important to note that measure theoretic framework of probability just discussed initially seems at odds with causal intuitions. In particular, the causal notion of random variables affecting one another is unnatural under the measure theoretic model in which all random variables are functions of a single random outcome selected from the sample space. Later we will see that this contradiction is superficial. Causal models are a special class of probabilistic models, with structured relationships between observed and unobserved (i.e., potential outcome) random variables. 

\section{Causal inference on two variables}\label{sec:two_vars}

The minimal causal model, and by far the most studied, is that of a binary treatment and a binary response. For the sake of simplicity, this is where we begin. We frame our discussion around the quintessential causal inference question: Does smoking cause lung cancer?

\subsection{Smoking and lung cancer}\label{sec:smoking_lung}

We model both smoking ($X$) and lung cancer ($Y$) as binary random variables on the square space as in Figure \ref{fig:mult}. In this example, the marginal probability of both smoking and lung cancer is $1/2$. A natural (but incorrect) approach one may take to quantify the effect of smoking on lung cancer is to estimate the \textbf{\textit{Average Observed Effect}}:
\begin{equation}\label{eq:aoe}
\text{AOE} \equiv \E[Y | X = 1] - \E[Y | X = 0]
\end{equation}   

\noindent For this particular example, $\text{AOE} = 3/4 - 1/4 = 1/2$. 

As a population quantity, the AOE must be estimated. Given a dataset of $n$ i.i.d. realizations of the bivariate random variable $(X,Y)$, one can compute the quantity: 
\begin{align*}
\widehat{\text{AOE}} &= \frac{1}{n_1}\sum_{i:X^{(i)} = 1}Y^{(i)} - \frac{1}{n_0}\sum_{i:X^{(i)} = 0}Y^{(i)} 
\end{align*}

\noindent  where $n_1 = \sum_i X_i$, $n_0 = n - n_1$, and $(X^{(i)}, Y^{(i)})$ denotes the $i^{\text{th}}$ sample (superscripts are used rather than subscripts to avoid confusion with notation introduced later). The law of large numbers ensures that $\widehat{\text{AOE}}$ converges to the true AOE as $n \to \infty$. Given enough samples, therefore, the AOE is estimable from the observed data. 

While the AOE is estimable from observable data, it does not generally correspond to any causal quantity. In particular, $\text{AOE} = 1/2$ implies nothing about how the incidence of lung cancer would change under an intervention in which cigarettes are eliminated from society altogether. Importantly, the difference in the conditional means of $Y$ could be completely or partially explained by a third confounding variable $Z$. 

\subsection{Potential outcome random variables}\label{sec:potential_outcomes}

The \textbf{\textit{potential outcomes}} of the Neyman-Rubin model (NRM) \cite{neyman, rubin_book} provide a language for causality distinct from statistical relationships between observed random variables. Following convention, we notate potential outcomes with subscripts and describe them intuitively as follows:

\begin{displayquote}
\begin{center}
$Y_x = \text{``$Y$ if $X$ had been $x$"}$
\end{center}
\end{displayquote}

\noindent If $Y_1 = 0$, then lung cancer would not be observed ($Y = 0$) in this particular individual if he had smoked, irrespective of whether or not he actually did smoke ($X = 1$). Potential outcomes are often described in the language of of ``alternate universes." If $X = 1$, then $Y_1$ is observed as $Y$. On the other hand, $Y_0$ is observed in the alternate universe which is identical to our universe in all respects except for the fact that $X = 0$. 

Though useful for intuition, this description of potential outcomes in terms of counterfactual realities is not stated in terms of probability spaces. In the present work, we emphasize that potential outcomes are familiar objects: random variables mapping $\omega \in \Omega$ to $\mathbb{R}$ defined on the same probability space as the random variables $X$ and $Y$. Potential outcomes are defined here according to a relationship with observable random variables. In the current example, the potential outcomes $Y_x$ are related to the observable random variable $Y$ by the following equation:
\begin{equation}\label{eq:marginalization_simple}
Y(\omega) = I_{0}(\omega) Y_0(\omega) + I_{1}(\omega) Y_1(\omega)  
\end{equation} 
\noindent where $I_{x}$ is the indicator random variable for the event $X = x$. This relationship, further generalized in Section \ref{sec:three_vars} and \ref{sec:generalization} by the \textbf{\textit{contraction}} operation, defines the essential structure of a causal model. 

One important feature of Equation \ref{eq:marginalization_simple} is that $Y_0$ and $Y_1$ are never simultaneously observed for a single $\omega$; $Y_1(\omega)$ is observed only when $X(\omega) = 1$ while $Y_0(\omega)$ is observed only when $X(\omega) = 0$. This observation is typically referred to as the \textbf{\textit{fundamental problem of causal inference}}. As a consequence of the fundamental problem of causal inference, there are generally many distinct sets of potential outcomes consistent with the observed random variables. For example the three distinct sets of potential outcomes $(\tilde{Y}_0, \tilde{Y}_1)$, $(\bar{Y}_0, \bar{Y}_1)$, and $(\check{Y}_0, \check{Y}_1)$ from Figure \ref{fig:tilde_bar_check} are all consistent with the observable random variables in Figure \ref{fig:mult}. This is achieved since $\tilde{Y}_0 = \bar{Y}_0 = \check{Y}_0$ on the pre-image $X^{-1}(0)$, while $\tilde{Y}_1 = \bar{Y}_1 = \check{Y}_1$ on the pre-image $X^{-1}(1)$. However, on $X^{-1}(0)$, the potential outcomes $\tilde{Y}_1$, $\bar{Y}_1$, and $\check{Y}_1$ may differ without altering the observed random variable $Y$. Likewise on $X^{-1}(1)$, the potential outcomes $\tilde{Y}_0$, $\bar{Y}_0$, and $\check{Y}_0$ may differ without altering the observed random variable $Y$.

\begin{figure}
\begin{center}
\begin{tikzpicture}[scale = 0.55]

\coordinate (CO) at (-3, 3);
\coordinate (CY0) at (5, 7);
\coordinate (CY1) at (5, -1);
\coordinate (CY02) at (13, 7);
\coordinate (CY12) at (13, -1);
\coordinate (CY03) at (21, 7);
\coordinate (CY13) at (21, -1);

\coordinate (nw) at (-3, 3);
\coordinate (ne) at (3, 3);
\coordinate (se) at (3, -3);
\coordinate (sw) at (-3, -3);


\fill[blue, opacity = 0.5] ($(CO) + (nw)$) -- ($(CO) + (ne)$) -- ($(CO) + (3, 0)$) -- ($(CO) + (-3, 0)$) -- cycle;

\fill[red, opacity = 0.5] ($(CO) + (ne)$) -- ($(CO) + (se)$) -- ($(CO) + (nw)$) -- cycle;

\fill[red, opacity = 0.5] ($(CY1) + (ne)$) -- ($(CY1) + (se)$) -- ($(CY1) + (nw)$) -- cycle;
\fill[red, opacity = 0.5] ($(CY0) + (ne)$) -- ($(CY0) + (se)$) -- ($(CY0) + (nw)$) -- cycle;

\fill[red, opacity = 0.5] ($(CY02) + (ne)$) -- ($(CY02) + (se)$) -- ($(CY02)$) -- ($(CY02) + (-3, 0)$) -- ($(CY02) + (nw)$) -- cycle;
\fill[red, opacity = 0.5] ($(CY12) + (ne)$) -- ($(CY12) + (3, 0)$) -- ($(CY12)$) -- ($(CY12) + (nw)$) -- cycle;

\fill[red, opacity = 0.5] ($(CY03) + (ne)$) -- ($(CY03) + (se)$) -- ($(CY03)$) -- ($(CY03) + (-3, 0)$) -- ($(CY03) + (nw)$) -- cycle;
\fill[red, opacity = 0.5] ($(CY13) + (ne)$) -- ($(CY13) + (se)$) -- ($(CY13) + (0, -3)$) -- ($(CY13)$) -- ($(CY13) + (nw)$) -- cycle;

\path ($(CY0) + (1, 2)$) node(F1) {$\tilde{Y}_0^{-1}(1)$}
($(CY1) + (1, 2)$) node(F1) {$\tilde{Y}_1^{-1}(1)$}
($(CY02) + (1, 2)$) node(F1) {$\bar{Y}_0^{-1}(1)$}
($(CY12) + (1, 2)$) node(F1) {$\bar{Y}_1^{-1}(1)$}
($(CY03) + (1, 2)$) node(F1) {$\check{Y}_0^{-1}(1)$}
($(CY13) + (1, 2)$) node(F1) {$\check{Y}_1^{-1}(1)$}
($(CO) + (0, -4)$) node(F1) {$\text{(a)}$}
($(CY1) + (0, -4)$) node(F1) {$\text{(b)}$}
($(CY12) + (0, -4)$) node(F1) {$\text{(c)}$}
($(CY13) + (0, -4)$) node(F1) {$\text{(d)}$};

\draw[ultra thick] ( $(CO) + (nw)$ ) -- ( $(CO) + (ne)$ ) -- ( $(CO) + (se)$ ) -- ( $(CO) +  (sw)$ ) -- cycle;
\draw ( $(CO) + (0, 3.5)$ ) node {\large $\Omega$};
\draw[ultra thick] ( $(CY0) + (nw)$ ) -- ( $(CY0) + (ne)$ ) -- ( $(CY0) + (se)$ ) -- ( $(CY0) +  (sw)$ ) -- cycle;
\draw ( $(CY0) + (0, 3.5)$ ) node {\large $\Omega$};
\draw[ultra thick] ( $(CY1) + (nw)$ ) -- ( $(CY1) + (ne)$ ) -- ( $(CY1) + (se)$ ) -- ( $(CY1) +  (sw)$ ) -- cycle;
\draw ( $(CY1) + (0, 3.5)$ ) node {\large $\Omega$};
\draw[ultra thick] ( $(CY02) + (nw)$ ) -- ( $(CY02) + (ne)$ ) -- ( $(CY02) + (se)$ ) -- ( $(CY02) +  (sw)$ ) -- cycle;
\draw ( $(CY02) + (0, 3.5)$ ) node {\large $\Omega$};
\draw[ultra thick] ( $(CY12) + (nw)$ ) -- ( $(CY12) + (ne)$ ) -- ( $(CY12) + (se)$ ) -- ( $(CY12) +  (sw)$ ) -- cycle;
\draw ( $(CY12) + (0, 3.5)$ ) node {\large $\Omega$};
\draw[ultra thick] ( $(CY03) + (nw)$ ) -- ( $(CY03) + (ne)$ ) -- ( $(CY03) + (se)$ ) -- ( $(CY03) +  (sw)$ ) -- cycle;
\draw ( $(CY03) + (0, 3.5)$ ) node {\large $\Omega$};
\draw[ultra thick] ( $(CY13) + (nw)$ ) -- ( $(CY13) + (ne)$ ) -- ( $(CY13) + (se)$ ) -- ( $(CY13) +  (sw)$ ) -- cycle;
\draw ( $(CY13) + (0, 3.5)$ ) node {\large $\Omega$};

\end{tikzpicture}
\end{center}
\caption{(a) The system of random variables $X$ and $Y$ from Figure \ref{fig:mult}. (b)-(d) Alternative sets of potential outcomes $(Y_0, Y_1)$ consistent with the observed random variables $X$ and $Y$.}\label{fig:tilde_bar_check}
\end{figure}
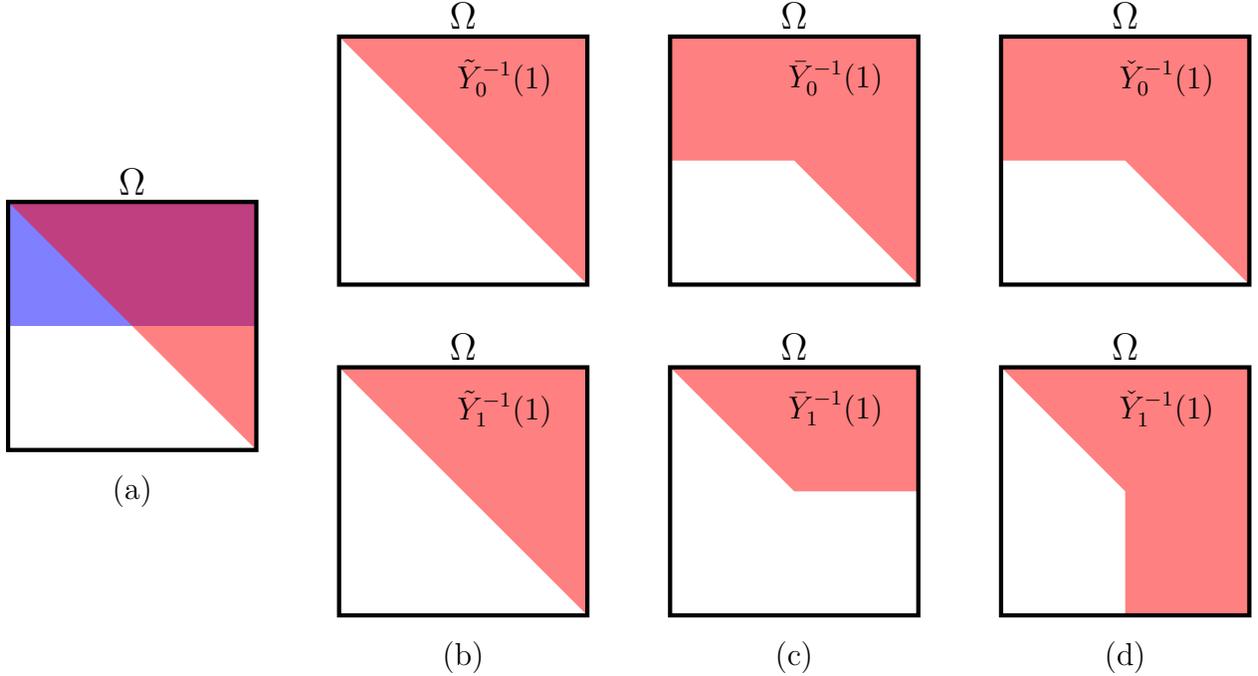

\subsection{Causal effects}\label{sec:causal_effects}

With potential outcomes, we can now define precise notions of causal effects by comparing the random variables $Y_0$ and $Y_1$. The definition below improves on the informal Definition \ref{def:informal_causality} by providing an unambiguous way to assess whether a binary random variable $X$ is causal for another random variable $Y$. 

\begin{defin}[Formal definition of causality]\label{def:causal_definition}
A binary random variable $X$ is causal for another random variable $Y$ (denoted $X \to Y$) if $Y_0(\omega) \ne Y_1(\omega)$ on a subset $F \in \mathcal{F}$ of nonzero measure.
\end{defin}

Referring to Figure \ref{fig:tilde_bar_check}, we see that if the set of potential outcomes are either $(\bar{Y}_0, \bar{Y}_1)$ or $(\check{Y}_0, \check{Y}_1)$, then we would conclude that $X$ is causal for $Y$. However, if the true set of potential outcomes is $(\tilde{Y}_0, \tilde{Y}_1)$, then we would conclude that $X$ is not causal for $Y$. Importantly, each set of potential outcomes is consistent with the observable random variables $X$ and $Y$. Irrespective of how large our sample is, we cannot conclude whether $X$ is causal for $Y$ from the observed data alone. Definition \ref{def:causal_definition} makes clear that the fundamental problem of causal inference is in direct conflict with any attempt to determine causal relationships from observed data. We develop this relationship further in Section \ref{sec:generalization}, where we generalize Definition \ref{def:causal_definition} and the fundamental problem of causal inference beyond the simple treatment and response paradigm discussed in the present section.

As was noted previously, it is typically not the case that we have complete knowledge of the probability space. Rather, we observe realizations of random variables. Through these observations we then try to infer their probability laws. Thus, it is important to have a definition of causality that depends only on distributional information. Perhaps the most important such metric is the \textbf{\textit{average causal effect}} (ACE):
\begin{equation}\label{eq:ace}
\text{ACE} \equiv \E[Y_1] - \E[Y_0]
\end{equation}
Referring again to Figure \ref{fig:tilde_bar_check}, we can compute the following:
\begin{align*}
\E[\tilde{Y}_1] - \E[\tilde{Y}_0] &= 1/2 - 1/2 = 0 \\ 
\E[\bar{Y}_1] - \E[\bar{Y}_0] &= 3/8 - 5/8 = -1/4 \\
\E[\check{Y}_1] - \E[\check{Y}_0] &= 5/8 - 5/8 = 0 
\end{align*}
\noindent If the underlying potential outcomes are $(\tilde{Y}_0, \tilde{Y}_1)$, then the ACE is zero, consistent with the observation that $X$ is not causal for $Y$. However, assuming the potential outcomes are $(\check{Y}_0, \check{Y}_1)$ yields an ACE which is also zero, despite the fact that $X$ is casual for $Y$. Finally, if the potential outcomes are $(\bar{Y}_0, \bar{Y}_1)$, then the ACE is $-1/4$. This is opposite in sign to the observable AOE which we found in Section \ref{sec:smoking_lung} to be $1/2$.

This example suggests that a nonzero ACE implies that $X$ is causal for $Y$ (although the inverse implication does not hold). For example, $\E[\bar{Y}_1] - \E[\bar{Y}_0] \ne 0$ and $X$ is causal for $Y$ under the set of potential outcomes $(\bar{Y}_1, \bar{Y}_0)$ in Figure \ref{fig:tilde_bar_check}. Corollary \ref{cor:ace_causation} below confirms this relationship for the case of binary $X$ and $Y$.

\begin{cor}[$(\text{ACE} \ne 0) \implies (X \to Y)$]\label{cor:ace_causation}
For binary $X$ and $Y$, if $\text{ACE} \ne 0$ then $X$ is causal for $Y$
\end{cor}

\begin{proof}
We first note that
\begin{align*}
(\text{ACE} \ne 0) \implies (\E[Y_1] \ne \E[Y_0]) \implies (P(Y_1^{-1}(1)) \ne P(Y_0^{-1}(1)))
\end{align*}
\noindent since $Y$ is assumed to be binary. Decomposing $P(Y_1^{-1}(1)$ and $P(Y_0^{-1}(1))$,
\begin{align*}
P(Y_1^{-1}(1)) &= P(Y_1^{-1}(1) \cap Y_0^{-1}(1)) + P(Y_1^{-1}(1) \setminus Y_0^{-1}(1)) \\
P(Y_0^{-1}(1)) &= P(Y_1^{-1}(1) \cap Y_0^{-1}(1)) + P(Y_0^{-1}(1) \setminus Y_1^{-1}(1))
\end{align*}
\noindent we notice
\begin{align*}
P(Y_1^{-1}(1) \setminus Y_0^{-1}(1)) \ne P(Y_0^{-1}(1) \setminus Y_1^{-1}(1))
\end{align*}
\noindent Therefore at least one of the events $Y_1^{-1}(1) \setminus Y_0^{-1}(1)$ or $Y_0^{-1}(1) \setminus Y_1^{-1}(1)$ must have nonzero measure. Therefore, by Definition \ref{def:causal_definition}, $X \to Y$. 
\end{proof}

As a brief side note, it is at least conceptually clear how one could generalize Definition \ref{def:causal_definition} to handle non-binary $X$. In particular, if $X(\Omega) \subseteq \mathbb{R}$ denotes the image of $X$, then $X$ is causal for $Y$ if the set of potential outcomes $\{Y_{x}\}_{x \in X(\Omega)}$ differ on a set $F \in \mathcal{F}$ of nonzero measure. However, when $X(\Omega)$ contains infinitely many elements, it may be the case that the potential outcomes $\{Y_{x}\}_{x \in X(\Omega)}$ differ on a subset of $F \in \mathcal{F}$ of nonzero measure, however this occurs for a subset $G \in X(\Omega)$ of zero measure. For instance, suppose $X(\Omega) = [0,1]$ and all of the potential outcomes $\{Y_x\}_{x \in [0,1]}$ are identical except for the potential outcome $Y_1$, which differs from all other potential outcomes on all of $\Omega$. For simplicity, we avoid such subtleties, considering exclusively finite discrete random variables in the present work, where the generalization of Definition \ref{def:causal_definition} is obvious.  

\subsection{Randomization}\label{sec:randomization}
We saw in the previous section a set of observable random variables $(X, Y)$ consistent with many sets of potential outcome random variables $(Y_0, Y_1)$ each implying different causal relationships. We also saw that determining whether $X$ is causal for $Y$ according to Definition \ref{def:causal_definition} is generally impossible since $Y_0$ and $Y_1$ are never simultaneously observable for any single $\omega \in \Omega$. Similarly, computing the ACE is generally impossible since it requires evaluating expectations of random variables $(Y_0, Y_1)$, which we only observe on incomplete and disjoint subsets of the sample space $\Omega$. 

However, when $X$ is independent of the potential outcomes $(Y_0, Y_1)$ (which we will denote as $X \independent (Y_0, Y_1)$ ) estimation of the average causal effect is possible. When this is the case, the following simple argument shows that $\text{AOE} = \text{ACE}$:

\begin{align*}
\text{AOE} &= \E[Y | X = 1] - \E[Y | X = 0] \\
&= \E[I_0 Y_0 + I_1 Y_1 | X = 1] - \E[I_0 Y_0 + I_1 Y_1 | X = 0] \\
&= \E[Y_1 | X = 1] - \E[Y_0 | X = 0] \\
&= \E[Y_1] - \E[Y_0] \\
&= \text{ACE}
\end{align*}

\noindent The second line follows from the first line by applying Equation \ref{eq:marginalization_simple}, which defines our causal model. The fourth line follows from the third line by our assumption that $X$ is independent of the potential outcome random variables. 

In a properly randomized experiment, it is often assumed that $X \independent (Y_0, Y_1)$. In the present section, we describe a measure theoretic model of the process of randomization, which takes advantage of the product measure construction described in Section \ref{sec:product_space}.

\begin{defin}[Experimental randomization of $X$]\label{def:randomization_definition}
Suppose $X$ and $Y$ are defined on a probability space $(\Omega, \mathcal{F}, P)$. An \textbf{\textit{experimental randomization of $X$}} produces a new probability space $(\tilde{\Omega}, \tilde{\mathcal{F}}, \tilde{P})$ and new random variables $\tilde{X}$ and $\tilde{Y}$ defined as follows:

\begin{align}
(\tilde{\Omega}, \tilde{\mathcal{F}}, \tilde{P}) &\equiv (\Omega \times \Omega_R, \mathcal{F} \times \mathcal{F}_R, P \times P_R) \\
\tilde{Y}(\tilde{\omega}) &\equiv \tilde{I}_0(\omega_R) Y_0(\omega) + \tilde{I}_1(\omega_R) Y_1(\omega) \\ 
\tilde{X}(\tilde{\omega}) &\equiv X_R(\omega_R) 
\end{align}

\noindent where $X_R$ is defined arbitrarily a new probability space $(\Omega_R, \mathcal{F}_R, P_R)$ such that $P_{X_R}(x) \in (0, 1)$ for all $x$.
\end{defin}

In an experimental randomization of $X$, the scientist replaces the ``naturally occurring" $X$ with an ``artificially generated" $X_R$, derived from an external source of randomization. An ideal (although unethical) randomized experiment to determine whether smoking causes lung cancer would allow the scientist to force individuals to smoke or not to smoke based on the outcome of a coin toss. Under experimental randomization, the choice to smoke is tied to an external source of randomness, and hence occurs altogether on a separate probability space $(\Omega_R, \mathcal{F}_R, P_R)$. The definition of $\tilde{Y}$ ensures that $\tilde{Y}$ responds to the randomized version ($X_R$) in the same way that it responded to the nonrandomized version ($X$). Defining the new observable random variables $\tilde{X}$ and $\tilde{Y}$ on the product space ensures that $X$ is independent of the potential outcome random variables $(Y_0, Y_1)$ as desired. 

As an example, suppose we experimentally randomize $X$ in the example from Figure \ref{fig:mult}, where the underlying potential outcomes $(Y_0, Y_1)$ are as in Figure \ref{fig:tilde_bar_check}(b). Suppose $X_R$ is defined on the probability space $(\Omega_R, \mathcal{F}_R, P_R) = ([0,1], \mathcal{B}_1, \mu_1)$, where $X_R^{-1}(0) = [0, 1/2]$ and $X_R^{-1}(1) = (1/2, 0]$. Then $P_{X_R}(1) = P_{X_R}(0) = 1/2$ as in the toss of an unbiased coin. Then the random variables $\tilde{X}$ and $\tilde{Y}$ live on the space $(\tilde{\Omega}, \tilde{\mathcal{F}}, \tilde{P}) = ([0,1]^3, \mathcal{B}_3, \mu_3)$, where $\mathcal{B}_3$ represents the Borel $\sigma$-algebra on $[0,1]^3$ and $\mu_3$ represents the three dimensional Lebesgue measure (equivalent to the common notion of volume). 

Figure \ref{fig:randomization} visualizes the randomization system $(\tilde{X}, \tilde{Y})$. We can compute the AOE on the randomized system as follows:

\begin{align*}
\text{AOE} &= \E[\tilde{Y} | \tilde{X} = 1] - \E[\tilde{Y} | \tilde{X} = 0] \\
&= 1/2 - 1/2 \\
&= 0 \\ 
&= \text{ACE}
\end{align*}

\noindent as expected. It may be instructive to verify that $\text{AOE} = \text{ACE}$ upon experimental randomization of $X$ for the three other sets of consistent potential outcomes in Figure \ref{fig:tilde_bar_check}, but geometric intuition should make it clear that this will always work. For the region $\tilde{X}^{-1}(0)$, we observe $Y_0 = \tilde{Y}$ over the entire cross section $\Omega$, so $\E[\tilde{Y} | X = 0] = \E[Y_0]$. The same reasoning makes it clear that $\E[\tilde{Y} | X = 1] = \E[Y_1]$. These two observations imply $\text{AOE} = \text{ACE}$ when $X$ is experimentally randomized. 

Theorem \ref{thm:random} describes an even more important consequence of experimental randomization. If $X$ is experimentally randomized, the probability law of potential outcomes can be deduced from observed conditional probability laws. 

\begin{thm}\label{thm:random}
Under experimental randomization of $X$,
\begin{align*}
P_{Y_x} = P_{\tilde{Y}|\tilde{X} = x}
\end{align*} 
\end{thm}

\begin{proof}
The proof follows from simply writing out the conditional probability explicitly:

\begin{align*}
P_{\tilde{Y} | \tilde{X} = x}(y) &\equiv \frac{\tilde{P}(\{\tilde{Y} = y\} \cap \{\tilde{X} = x\})}{\tilde{P}(\tilde{X} = x)} \\  
&= \frac{\tilde{P}\left(\left\{ \{ Y_0^{-1}(y) \times X_R^{-1}(0)\} \cup \{ Y_1^{-1}(y) \times X_R^{-1}(1)\} \right\}  \cap \{\Omega \times X_R^{-1}(x)\}\right)}{\tilde{P}(\{\Omega \times X_R^{-1}(x))\}} \\  
&= \frac{\tilde{P}(\{Y_{x}^{-1}(y) \times X_R^{-1}(x))\})}{\tilde{P}(\{\Omega \times X_R^{-1}(x))\}} \\  
&= \frac{P(Y_{x}^{-1}(y))P_R(X_R^{-1}(x))}{P(\Omega) P_R(X_R^{-1}(x))} \\
&=  P(Y_{x}^{-1}(y)) \\
&=  P_{Y_x}(y) 
\end{align*}
\end{proof}

The discussion of the present section makes clear why randomization is such a powerful technique. In a properly randomized system, true causal quantities such as the ACE can computed from observed data. However, it is important to recognize the shortcomings of experimental randomization. First, experimental randomization is still inadequate for the purposes of uncovering causality in situations like Figure \ref{fig:tilde_bar_check}d; although $X$ is causal for $Y$ according to Definition \ref{def:causal_definition}, the probability laws $P_{Y_0}$ and $P_{Y_1}$ are identical. Second, the conditions of Definition \ref{def:randomization_definition} are very strict. Beyond just ensuring that $X \independent (Y_0, Y_1)$, experimental randomization requires that $X_R$ can behave as a substitute for $X$ in Equation \ref{eq:marginalization_simple}. For instance, if being involved in a randomized trial induces behavior that has some effect on lung cancer (i.e., cognizance of enrollment in a lung cancer trial may cause participants to pursue a healthier lifestyle), we cannot expect the causal effects computed from the randomized trial to reflect the causal effect of smoking ``in the wild."

\begin{figure}
\begin{center}
\begin{tikzpicture}[scale = 0.55]

\coordinate (CO) at (-3, 3);
\coordinate (C0) at (5, 3);
\coordinate (C1) at (13, 3);

\coordinate (nw) at (-3, 3);
\coordinate (ne) at (3, 3);
\coordinate (se) at (3, -3);
\coordinate (sw) at (-3, -3);
\coordinate (shift) at (-2, -2);
\coordinate (hshift) at (-1, -1);
\coordinate (mcshift) at (-.1, -.1);


\fill[blue, opacity = 0.5] ($(CO) + (nw)$) -- ($(CO) + (ne)$) -- ($(CO) + (ne) + (hshift)$) -- ($(CO) + (nw) + (hshift)$) -- cycle;
\fill[blue, opacity = 0.5] ($(CO) + (ne)$) -- ($(CO) + (se)$) -- ($(CO) + (se) + (hshift)$) -- ($(CO) + (ne) + (hshift)$) -- cycle;

\fill[blue, opacity = 0.5] ($(CO) + (nw) + (hshift) + (0, -2)$) -- ($(CO) + (se) + (hshift) + (-2, 0)$) -- ($(CO) + (sw) + (hshift)$) -- cycle;

\fill[red, opacity = 0.5] ($(CO) + (nw) + (shift)$) -- ($(CO) + (ne) + (shift)$) -- ($(CO) + (se) + (shift)$)-- cycle;

\fill[red, opacity = 0.5] ($(CO) + (nw) + (shift)$) -- ($(CO) + (ne) + (shift)$) -- ($(CO) + (ne) + (hshift)$) -- ($(CO) + (nw) + (hshift)$) -- cycle;
\fill[red, opacity = 0.5] ($(CO) + (ne) + (shift)$) -- ($(CO) + (se) + (shift)$) -- ($(CO) + (se) + (hshift)$) -- ($(CO) + (ne) + (hshift)$) -- cycle;

\fill[red, opacity = 0.5] ($(CO) + (nw)$) -- ($(CO) + (ne)$) -- ($(CO) + (ne) + (hshift)$) -- ($(CO) + (nw) + (hshift)$) -- cycle;
\fill[red, opacity = 0.5] ($(CO) + (ne)$) -- ($(CO) + (se)$) -- ($(CO) + (se) + (hshift)$) -- ($(CO) + (ne) + (hshift)$) -- cycle;

\draw[dashed, black] ($(CO) + (nw)$) -- ($(CO) + (se)$);
\draw[dashed, black] ($(CO) + (nw) + (hshift)$) -- ($(CO) + (se) + (hshift)$);
\draw[black] ($(CO) + (nw) + (hshift)$) -- ($(CO) + (ne) + (hshift)$);
\draw[black] ($(CO) + (ne) + (hshift)$) -- ($(CO) + (se) + (hshift)$);
\draw[dashed, black] ($(CO) + (nw) + (hshift)$) -- ($(CO) + (sw) + (hshift)$);
\draw[dashed, black] ($(CO) + (sw) + (hshift)$) -- ($(CO) + (se) + (hshift)$);
\draw[dashed, black] ($(CO) + (sw) + (shift)$) -- ($(CO) + (sw)$);
\draw[dashed, black] ($(CO) + (sw)$) -- ($(CO) + (nw)$);
\draw[dashed, black] ($(CO) + (sw)$) -- ($(CO) + (se)$);

\draw[black] ($(CO) + (nw) + (shift)$) -- ($(CO) + (se) + (shift)$);
\draw[black] ($(CO) + (nw) + (shift)$) -- ($(CO) + (ne) + (shift)$);
\draw[black] ($(CO) + (ne) + (shift)$) -- ($(CO) + (se) + (shift)$);

\fill[red, opacity = 0.5] ($(C0) + (nw)$) -- ($(C0) + (ne)$) -- ($(C0) + (se)$) -- cycle;

\fill[blue, opacity = 0.5] ($(C1) + (nw)$) -- ($(C1) + (ne)$) -- ($(C1) + (se)$) -- ($(C1) + (sw)$) -- cycle;
\fill[red, opacity = 0.5] ($(C1) + (nw)$) -- ($(C1) + (ne)$) -- ($(C1) + (se)$) -- cycle;

\draw[-, ultra thick] ($(CO) + (se) + (shift) + (1, -1)$) -- ($(CO) + (se) + (1, -1)$);
\draw[-, ultra thick] ($(CO) + (se) + (shift) + (1, -1 + 0.2)$) -- ($(CO) + (se) + (shift) + (1, -1 - 0.2)$) node[below] 
{\large $0$};
\draw[-, ultra thick] ($(CO) + (se) + (1, -1 + 0.2)$) -- ($(CO) + (se) + (1, -1 - 0.2)$) node[below] 
{\large $1$};

\fill[blue, opacity = 0.5]  ($(CO) + (se) + (hshift) + (1, -1 + 0.2)$) -- ($(CO) + (se) + (hshift) + (1, -1 - 0.2)$) -- ($(CO) + (se) + (1, -1 - 0.2)$) -- ($(CO) + (se) + (1, -1 + 0.2)$) -- cycle;
\path ($(CO) + (se) + (hshift) + (2, -2)$) node(O1) {\large $\Omega_R$}; 
\path ($(CO) + (sw) + (hshift) + (2, -2)$) node(O1) {\large $\Omega$};

\draw[ultra thick] ( $(CO) + (nw) + (shift)$ ) -- ( $(CO) + (nw)$ ) -- ( $(CO) + (ne)$ ) -- ( $(CO) +  (ne) + (shift)$ ) -- cycle;
\draw[ultra thick] ( $(CO) + (ne) + (shift)$ ) -- ( $(CO) + (ne)$ ) -- ( $(CO) + (se)$ ) -- ( $(CO) +  (se) + (shift)$ ) -- cycle;
\draw[ultra thick] ( $(CO) + (nw) + (shift)$ ) -- ( $(CO) + (ne) + (shift)$ ) -- ( $(CO) + (se) + (shift)$ ) -- ( $(CO) +  (sw) + (shift)$ ) -- cycle;
\draw ( $(CO) + (0, 3.6)$ ) node {\large $\tilde{\Omega} = \Omega \times \Omega_R$};
\draw[ultra thick] ( $(C0) + (nw)$ ) -- ( $(C0) + (ne)$ ) -- ( $(C0) + (se)$ ) -- ( $(C0) + (sw)$ ) -- cycle;
\draw ( $(C0) + (0, 3.6)$ ) node {\large $\Omega \times 0$};
\draw[ultra thick] ( $(C1) + (nw)$ ) -- ( $(C1) + (ne)$ ) -- ( $(C1) + (se)$ ) -- ( $(C1) + (sw)$ ) -- cycle;
\draw ( $(C1) + (0, 3.6)$ ) node {\large $\Omega \times 1$};

\end{tikzpicture}
\end{center}
\caption{Experimental randomization of $X$ induces a product space structure.}\label{fig:randomization}
\end{figure}
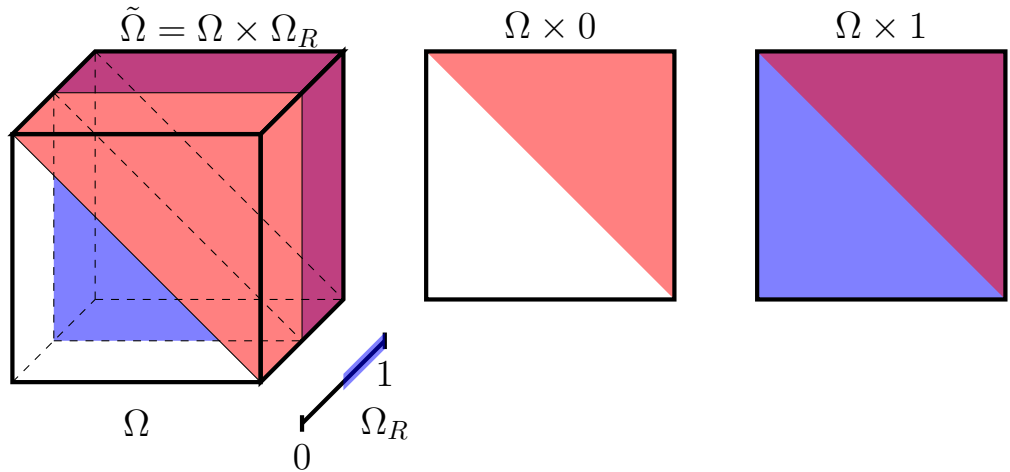

\section{Causal inference on three variables}\label{sec:three_vars}

Several new concepts in causality arise in systems of three observable variables. As such, in this section we study the simplest three-variable system: three binary random variables. We add to our running example of smoking ($X$) and lung cancer ($Y$) a third binary random variable $Z$ representing exercise habits. $Z = 0$ indicates a low level of exercise while $Z = 1$ indicates a high level of exercise. One could imagine exercise habits influencing both lung cancer outcomes and smoking choices. 

\subsection{A comment on notation}
In previous sections we only needed a single subscript to specify potential outcomes. For instance, $Y_1$ implicitly referred to the potential outcome ``$Y$ had $X$ been 1." The potential outcome $Y_0$ from previous sections will now be denoted $Y_{X = 0}$ in order to distinguish it from the potential outcome $Y_{Z = 0}$. Further, the potential outcome ``$Y$ had $X$ been 0 and $Z$ been 1" will be denoted $Y_{(X, Z) = (0, 1)}$. 

\subsection{Contraction}\label{sec:marginalization}

Equation \ref{eq:marginalization_simple} specifies the relationship between potential outcome random variables $(Y_{X = 0}, Y_{X = 1})$ and observable random variables $X$ and $Y$. In the case of three random variables, we might naturally generalize Equation \ref{eq:marginalization_simple} as follows

\begin{align*}
Y(\omega) = \sum_{x}\sum_{z} I_{(X,Z)=(x,z)}(\omega)Y_{(X,Z) = (x,z)}(\omega)
\end{align*} 

\noindent Since Equation \ref{eq:marginalization_simple} must still hold, we have the following equality:

\begin{equation}\label{eq:equality}
\sum_x I_{X = x}(\omega)Y_{X = x}(\omega) = \sum_{x} \sum_{z} I_{(X, Z) = (x, z)}(\omega)Y_{(X, Z) = (x,z)}(\omega)
\end{equation}

\noindent Together with the observation that $I_{(X, Z) = (x, z)}(\omega) = I_{X = x}(\omega)I_{Z = z}(\omega)$, Equation \ref{eq:equality} implies the following relationship between the double-subscripted potential outcomes $Y_{(X, Z) = (x, z)}$ and the single-subscripted potential outcomes $Y_{X = x}$:

\begin{align*}
Y_{X = x}(\omega) = \sum_{z} I_{Z = z}(\omega)Y_{(X, Z) = (x, z)}(\omega)
\end{align*}

\noindent Similar reasoning suggests the following relationship for the potential outcomes $Y_{Z = z}$:

\begin{align*}
Y_{Z = z}(\omega) = \sum_{x} I_{X = x}(\omega)Y_{(X, Z) = (x, z)}(\omega)
\end{align*}

In this manner, any single-subscripted potential outcome may be derived from double-subscripted potential outcomes and observable random variables: by summing over the subscript to be removed and multiplying by the corresponding indicator random variables. We will refer to this operation as contraction, due to its similarity to tensorial contraction. For instance, the set of potential outcomes $\{Y_{X = x}\}$ are obtained from the potential outcomes $\{Y_{(X, Z) = (x, z)}\}$ by ``contraction over $z$." The observable random variable $Y$ can be obtained by ``contracting $\{Y_{(X, Z) = (x, z)}\}$ over $x$ and $z$" or equivalently ``contracting $\{Y_{x}\}$ over $x$." Thus, the simple relationship in Equation \ref{eq:marginalization_simple} represents a contraction. We will formalize and generalize the notion of contraction in Section \ref{sec:generalization}.

\subsection{Joint causality}\label{sec:joint_causality}

In a system of three binary observable random variables $(X, Y, Z)$, Definition \ref{def:causal_definition} is still applicable to pairs of variables. For instance, $X$ is causal for $Y$ if the potential outcomes $\{Y_{X =  x}\}$, obtained by contracting $\{Y_{(X, Z) = (x, z)}\}$ over $z$, are different on a subset of the sample space of nonzero measure. Similarly, one can assess if $Z$ is causal for $Y$ by examining the potential outcomes obtained by contracting over $x$. 

However, it is also possible for $X$ and $Z$ to affect $Y$ in a way not fully explained by their individual effects on $Y$. Figure \ref{fig:joint_causality} displays a particularly pronounced example. Here, neither $X$ nor $Z$ is causal for $Y$ alone according to Definition \ref{def:causal_definition}. This is because $Y_{X = 0} = Y_{X = 1}$ and $Y_{Z = 0} = Y_{Z = 1}$ for all $\omega \in \Omega$. In fact, all single-subscripted potential outcomes $\{Y_{X = x}\}$ and $\{Y_{Z = z}\}$ equal to zero on all of $\Omega$. For example:
\begin{align*}
Y_{X = 0}(\omega) &= I_{Z = 0}(\omega) Y_{(X, Z) = (0, 0)}(\omega) + I_{Z = 1}(\omega) Y_{(X, Z) = (0, 1)}(\omega) \\
&= 0
\end{align*}
\noindent for all $\omega \in \Omega$. This is because $Y_{X = 0} = Y_{(X, Z) = (0, 0)}$ on $Z^{-1}(0)$, where $Y_{(X, Z) = (0, 0)} = 0$. Likewise $Y_{X = 0} = Y_{(X, Z) = (0, 1)}$ on $Z^{-1}(1)$, where $Y_{(X, Z) = (0, 1)} = 0$. Similar calculations can be done for each of the other three single-indexed potential outcomes $Y_{X = 1}$, $Y_{Z = 0}$, and $Y_{Z = 1}$, and one can confirm that each of these potential outcomes is identically zero on all of $\Omega$.

However, the double-subscripted potential outcomes $\{Y_{(X,Z) = (x, z)}\}$ differ from each other on a subset of $\Omega$ of measure one. This is because for all $\omega \in \Omega$ (excluding the measure zero subset along the vertical and horizontal mid-line of $\Omega$), exactly one double-subscripted potential outcome $Y_{(X, Z) = (x, z)}$ is equal to one, with each of the other three equal to zero. In this example, we will say that $X$ and $Z$ are jointly causal for $Y$. 

\begin{figure}
\begin{center}
\begin{tikzpicture}[scale = 0.55]
\coordinate (CX) at (5, 7);
\coordinate (CZ) at (13, 7);
\coordinate (CY) at (21, 7);
\coordinate (CY00) at (1, 0);
\coordinate (CY01) at (9, 0);
\coordinate (CY10) at (1, -7);
\coordinate (CY11) at (9, -7);
\coordinate (CYX0) at (17, 0);
\coordinate (CYX1) at (25, 0);
\coordinate (CYZ0) at (17, -7);
\coordinate (CYZ1) at (25, -7);

\coordinate (nw) at (-3, 3);
\coordinate (ne) at (3, 3);
\coordinate (se) at (3, -3);
\coordinate (sw) at (-3, -3);

\fill[blue, opacity = 0.5] ($(CX) + (nw)$) -- ($(CX) + (-3, 0)$) -- ($(CX) + (3, 0)$) -- ($(CX) + (ne)$) -- cycle;
\fill[green, opacity = 0.5] ($(CZ) + (nw)$) -- ($(CZ) + (0, 3)$) -- ($(CZ) + (0, -3)$) -- ($(CZ) + (sw)$) -- cycle;

\fill[red, opacity = 0.5] ($(CY00) + (nw)$) -- ($(CY00) + (0, 3)$) -- ($(CY00)$) -- ($(CY00) + (-3,0)$) -- cycle;
\fill[red, opacity = 0.5] ($(CY01) + (0,3)$) -- ($(CY01) + (ne)$) -- ($(CY01) + (3, 0)$) -- ($(CY01)$) -- cycle;
\fill[red, opacity = 0.5] ($(CY10) + (-3, 0)$) -- ($(CY10)$) -- ($(CY10) + (0,-3)$) -- ($(CY10) + (sw)$) -- cycle;
\fill[red, opacity = 0.5] ($(CY11)$) -- ($(CY11) + (3,0)$) -- ($(CY11) + (se)$) -- ($(CY11) + (0,-3)$) -- cycle;

\path ($(CX) + (0, 1.5)$) node(F1) {$X^{-1}(1)$}
($(CZ) + (-1.5, 1.5)$) node(F1) {$Z^{-1}(1)$}
($(CY) + (0, 1.5)$) node(F1) {$Y^{-1}(0)$}
($(CYX0) + (0, 1.5)$) node(F1) {$Y_{X = 0}^{-1}(0)$}
($(CYX1) + (0, 1.5)$) node(F1) {$Y_{X = 1}^{-1}(0)$}
($(CYZ0) + (0, 1.5)$) node(F1) {$Y_{Z = 0}^{-1}(0)$}
($(CYZ1) + (0, 1.5)$) node(F1) {$Y_{Z = 1}^{-1}(0)$};
\node [rotate=-45]  at ($(CY00) + (-1.5, 1.5)$) {\tiny $Y_{(X, Z) = (0, 0)}^{-1}(1)$};
\node [rotate=-45]  at ($(CY01) + (1.5, 1.5)$) {\tiny $Y_{(X, Z) = (0, 1)}^{-1}(1)$};
\node [rotate=-45]  at ($(CY10) + (-1.5, -1.5)$) {\tiny $Y_{(X, Z) = (1, 0)}^{-1}(1)$};
\node [rotate=-45]  at ($(CY11) + (1.5, -1.5)$) {\tiny $Y_{(X, Z) = (1, 1)}^{-1}(1)$};

\draw[ultra thick] ( $(CX) + (nw)$ ) -- ( $(CX) + (ne)$ ) -- ( $(CX) + (se)$ ) -- ( $(CX) +  (sw)$ ) -- cycle;
\draw ( $(CX) + (0, 3.5)$ ) node {\large $\Omega$};

\draw[ultra thick] ( $(CZ) + (nw)$ ) -- ( $(CZ) + (ne)$ ) -- ( $(CZ) + (se)$ ) -- ( $(CZ) +  (sw)$ ) -- cycle;
\draw ( $(CZ) + (0, 3.5)$ ) node {\large $\Omega$};

\draw[ultra thick] ( $(CY) + (nw)$ ) -- ( $(CY) + (ne)$ ) -- ( $(CY) + (se)$ ) -- ( $(CY) +  (sw)$ ) -- cycle;
\draw ( $(CY) + (0, 3.5)$ ) node {\large $\Omega$};

\draw[ultra thick] ( $(CY00) + (nw)$ ) -- ( $(CY00) + (ne)$ ) -- ( $(CY00) + (se)$ ) -- ( $(CY00) +  (sw)$ ) -- cycle;
\draw ( $(CY00) + (0, 3.5)$ ) node {\large $\Omega$};

\draw[ultra thick] ( $(CY01) + (nw)$ ) -- ( $(CY01) + (ne)$ ) -- ( $(CY01) + (se)$ ) -- ( $(CY01) +  (sw)$ ) -- cycle;
\draw ( $(CY01) + (0, 3.5)$ ) node {\large $\Omega$};

\draw[ultra thick] ( $(CY10) + (nw)$ ) -- ( $(CY10) + (ne)$ ) -- ( $(CY10) + (se)$ ) -- ( $(CY10) +  (sw)$ ) -- cycle;
\draw ( $(CY10) + (0, 3.5)$ ) node {\large $\Omega$};

\draw[ultra thick] ( $(CY11) + (nw)$ ) -- ( $(CY11) + (ne)$ ) -- ( $(CY11) + (se)$ ) -- ( $(CY11) +  (sw)$ ) -- cycle;
\draw ( $(CY11) + (0, 3.5)$ ) node {\large $\Omega$};

\draw[ultra thick] ( $(CYX0) + (nw)$ ) -- ( $(CYX0) + (ne)$ ) -- ( $(CYX0) + (se)$ ) -- ( $(CYX0) +  (sw)$ ) -- cycle;
\draw ( $(CYX0) + (0, 3.5)$ ) node {\large $\Omega$};

\draw[ultra thick] ( $(CYX1) + (nw)$ ) -- ( $(CYX1) + (ne)$ ) -- ( $(CYX1) + (se)$ ) -- ( $(CYX1) +  (sw)$ ) -- cycle;
\draw ( $(CYX1) + (0, 3.5)$ ) node {\large $\Omega$};

\draw[ultra thick] ( $(CYZ0) + (nw)$ ) -- ( $(CYZ0) + (ne)$ ) -- ( $(CYZ0) + (se)$ ) -- ( $(CYZ0) +  (sw)$ ) -- cycle;
\draw ( $(CY11) + (0, 3.5)$ ) node {\large $\Omega$};

\draw[ultra thick] ( $(CYZ1) + (nw)$ ) -- ( $(CYZ1) + (ne)$ ) -- ( $(CYZ1) + (se)$ ) -- ( $(CYZ1) +  (sw)$ ) -- cycle;
\draw ( $(CYZ1) + (0, 3.5)$ ) node {\large $\Omega$};

\end{tikzpicture}
\end{center}
\caption{A system of three random variables $X$, $Y$, and $Z$ for which $X$ and $Z$ are jointly causal for $Y$, but neither is individually causal for $Y$.}\label{fig:joint_causality}
\end{figure}
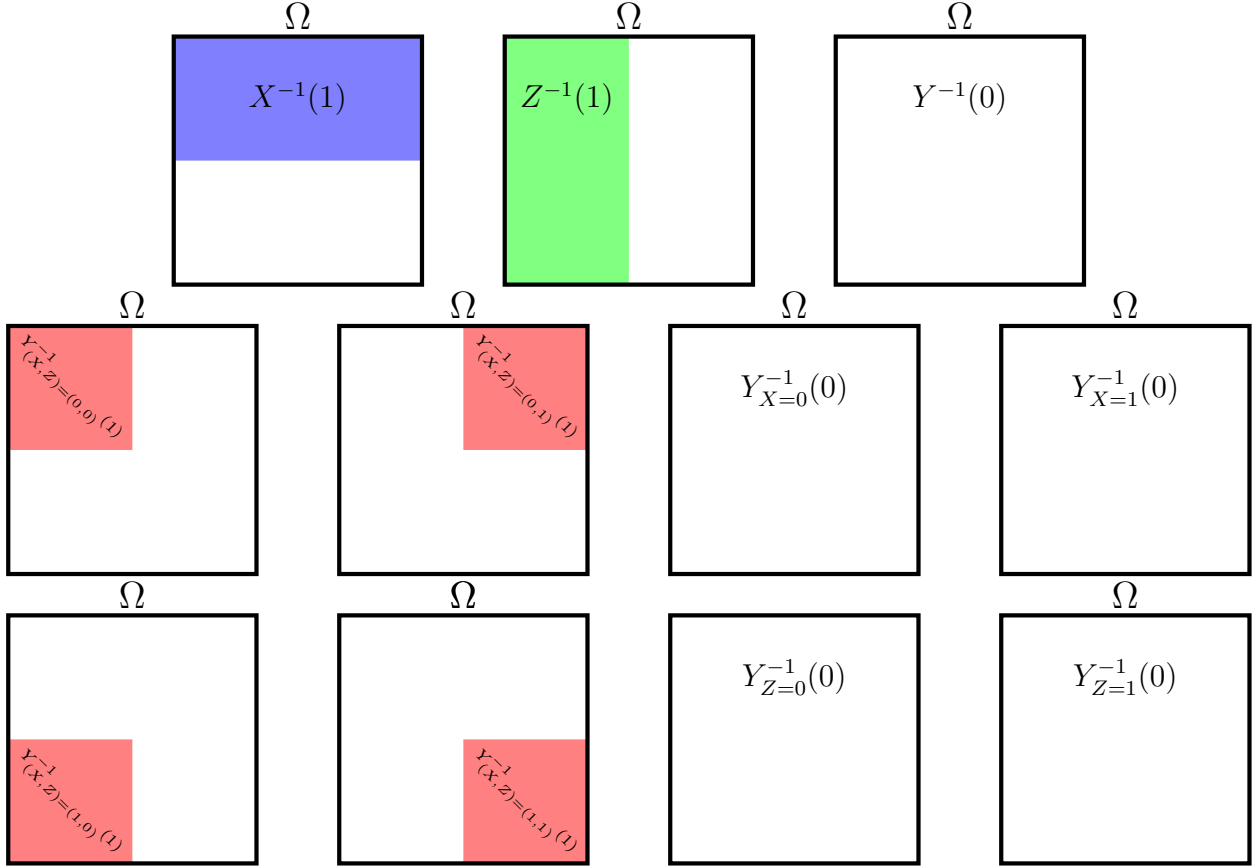

Before precisely defining joint causality, we first recognize that Definition \ref{def:causal_definition} can also apply to causal relationships between observable and potential outcome random variables. Noting that $Y_{X = x}$ is itself a random variable, we can conclude that $Z$ is causal for $Y_{X = x}$ if $Y_{(X, Z) = (x, 0)} \ne Y_{(X, Z) = (x, 1)}$ on a subset $F \in \mathcal{F}$ of nonzero measure. Intuitively, if $z$ is causal for $Y_{X = x}$, the effect $X$ has on $Y$ is modified by the value $Z$. 

However, $Z$ being causal for $Y_{X = x}$ alone does not capture the notion of joint causality. For example, consider the set of potential outcomes $\{Y_{(X, Z) = (x, z)}\}$ displayed in Figure \ref{fig:falsely_joint}. In this case, $Z$ is causal for $Y_{X = 0}$ since $Y_{(X, Z) = (0, 0)}$ and $Y_{(X, Z) = (0, 1)}$ differ on all of $\Omega$. Similarly, $Z$ is also causal for $Y_{X = 1}$. However, the potential outcomes do not depend on the $x$ subscript at all: the value of $Y$ can be determined by $\omega$ and the $z$ subscript alone. 

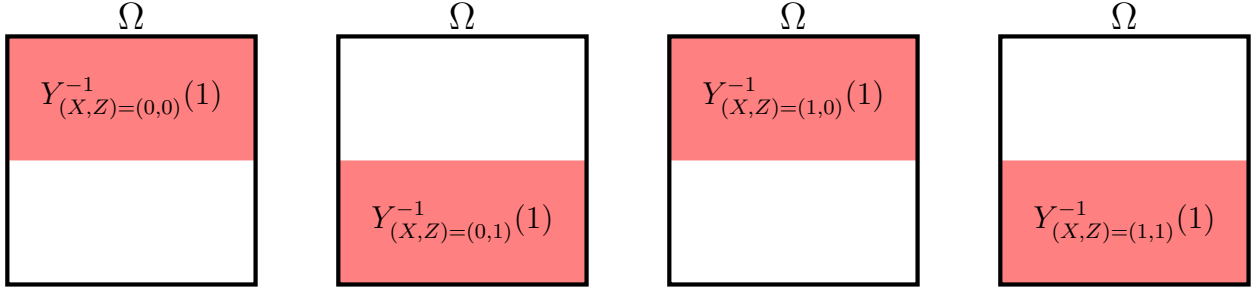
\begin{figure}
\begin{center}
\begin{tikzpicture}[scale = 0.55]
\coordinate (CY00) at (1, 0);
\coordinate (CY01) at (9, 0);
\coordinate (CY10) at (17, 0);
\coordinate (CY11) at (25, 0);

\coordinate (nw) at (-3, 3);
\coordinate (ne) at (3, 3);
\coordinate (se) at (3, -3);
\coordinate (sw) at (-3, -3);

\fill[red, opacity = 0.5] ($(CY00) + (nw)$) -- ($(CY00) + (ne)$) -- ($(CY00) + (3,0)$) -- ($(CY00) + (-3,0)$) -- cycle;
\fill[red, opacity = 0.5] ($(CY01) + (-3, 0)$) -- ($(CY01) + (sw)$) -- ($(CY01) + (se)$) -- ($(CY01) + (3, 0)$) -- cycle;
\fill[red, opacity = 0.5] ($(CY10) + (nw)$) -- ($(CY10) + (ne)$) -- ($(CY10) + (3,0)$) -- ($(CY10) + (-3,0)$) -- cycle;
\fill[red, opacity = 0.5] ($(CY11) + (-3, 0)$) -- ($(CY11) + (sw)$) -- ($(CY11) + (se)$) -- ($(CY11) + (3, 0)$) -- cycle;

\node [rotate=0]  at ($(CY00) + (0, 1.5)$) {$Y_{(X, Z) = (0, 0)}^{-1}(1)$};
\node [rotate=0]  at ($(CY01) + (0, -1.5)$) {$Y_{(X, Z) = (0, 1)}^{-1}(1)$};
\node [rotate=0]  at ($(CY10) + (0, 1.5)$) {$Y_{(X, Z) = (1, 0)}^{-1}(1)$};
\node [rotate=0]  at ($(CY11) + (0, -1.5)$) {$Y_{(X, Z) = (1, 1)}^{-1}(1)$};

\draw[ultra thick] ( $(CY00) + (nw)$ ) -- ( $(CY00) + (ne)$ ) -- ( $(CY00) + (se)$ ) -- ( $(CY00) +  (sw)$ ) -- cycle;
\draw ( $(CY00) + (0, 3.5)$ ) node {\large $\Omega$};

\draw[ultra thick] ( $(CY01) + (nw)$ ) -- ( $(CY01) + (ne)$ ) -- ( $(CY01) + (se)$ ) -- ( $(CY01) +  (sw)$ ) -- cycle;
\draw ( $(CY01) + (0, 3.5)$ ) node {\large $\Omega$};

\draw[ultra thick] ( $(CY10) + (nw)$ ) -- ( $(CY10) + (ne)$ ) -- ( $(CY10) + (se)$ ) -- ( $(CY10) +  (sw)$ ) -- cycle;
\draw ( $(CY10) + (0, 3.5)$ ) node {\large $\Omega$};

\draw[ultra thick] ( $(CY11) + (nw)$ ) -- ( $(CY11) + (ne)$ ) -- ( $(CY11) + (se)$ ) -- ( $(CY11) +  (sw)$ ) -- cycle;
\draw ( $(CY11) + (0, 3.5)$ ) node {\large $\Omega$};

\end{tikzpicture}
\end{center}
\caption{A set of potential outcomes $\{Y_{(X, Z) = (x, z)}\}_{(x, z)}$ for which $Z$ is causal for $Y_{X = 0}$ and $Y_{X = 1}$, however $X$ and $Z$ are not jointly causal for $Y$.}\label{fig:falsely_joint}
\end{figure}

To ensure that we exclude scenarios like that in Figure \ref{fig:falsely_joint}, we define joint causality as follows: 

\begin{defin}[Joint causality]\label{def:joint_causality}
Two binary random variables $X$ and $Z$ are said to be \textbf{\textit{jointly causal}} for a third random variable $Y$ if both of the following hold:
\begin{enumerate}[(i)]
\item $Z$ is causal for $Y_{X = x}$ for some $x$. 
\item $X$ is causal for $Y_{Z = z}$ for some $z$. 
\end{enumerate}
\end{defin}

\noindent According to Definition \ref{def:joint_causality}, $X$ and $Z$ are jointly causal for $Y$ in Figure \ref{fig:joint_causality}, but not jointly causal in Figure \ref{fig:falsely_joint}.

As with Definition \ref{def:causal_definition}, some generalizations of Definition \ref{def:joint_causality} obvious while others are not. For one, Definition \ref{def:joint_causality} does not at all depend on $X$ and $Z$ being binary; the definition is equally applicable to any finite discrete $X$ and $Z$. When either $X$ or $Z$ are continuous, we encounter the same subtleties as in Definition \ref{def:causal_definition}. We can also imagine the definition of joint causality applying to sets of more than two random variables. For three random variables $A$, $B$ and $C$ to be jointly causal for a fourth random variable $Y$, we require i) $A$ to be causal for $Y_{(B, C) = (b, c)}$ for some $(b, c)$ ii) $B$ to be causal for $Y_{(A, C) = (a, c)}$ for some $(a, c)$ and iii) $C$ to be causal for $Y_{(A, B) = (a, b)}$ for some $(a, b)$. The generalization to four or more finite discrete random variables is now straightforward. 

\subsection{Joint randomization}\label{sec:joint_randomization}
In Theorem \ref{thm:random}, we saw that experimental randomization of $X$ allowed us to infer the distribution of the potential outcome $Y_{X = x}$ from the distribution of the observable random variable $\tilde{Y} | \tilde{X} = x$. In the present section, we show how one can simultaneously randomize $X$ and $Z$ to infer the distribution of the potential outcomes $Y_{(X, Z) = (x, z)}$. This procedure of simultaneous randomization, detailed in Definition \ref{def:joint_randomization}, is a natural extension of the procedure detailed in Definition \ref{def:randomization_definition}.

\begin{defin}[Joint experimental randomization of $X$ and $Z$]\label{def:joint_randomization}
Suppose $X$, $Y$, and $Z$ are defined on a probability space $(\Omega, \mathcal{F}, P)$. A \textbf{\textit{joint experimental randomization of $X$ and $Z$}} produces a new probability space $(\tilde{\Omega}, \tilde{\mathcal{F}}, \tilde{P})$ and new random variables $\tilde{X}$, $\tilde{Y}$, and $\tilde{Z}$ defined as follows:

\begin{align}
(\tilde{\Omega}, \tilde{\mathcal{F}}, \tilde{P}) &\equiv (\Omega \times \Omega_R, \mathcal{F} \times \mathcal{F}_R, P \times P_R) \\
\tilde{Y}(\tilde{\omega}) &\equiv \sum_x \sum_z \tilde{I}_{(\tilde{X}, \tilde{Z}) = (x, z)}(\omega_R) Y_{(X, Z) = (x, z)}(\omega) \\ 
\tilde{X}(\tilde{\omega}) &\equiv X_R(\omega_R) \\ 
\tilde{Z}(\tilde{\omega}) &\equiv Z_R(\omega_R) 
\end{align}

\noindent where $X_R$ and $Z_R$ are defined arbitrarily on a probability spaces $(\Omega_R, \mathcal{F}_R, P_R)$ such that $P_{(X_R, Z_R)}(x, z) \in (0, 1)$ for all $(x, y)$.
\end{defin}

In the definition of joint experimental randomization, we do not require $X$ and $Z$ to be randomized on separate probability spaces. In other words, joint experimental randomization of $X$ and $Z$ does not necessarily require $\tilde{X}$ and $\tilde{Z}$ to be independent of each other. Of course, randomizing $X$ and $Z$ on separate probability spaces $(\Omega_R^X, \mathcal{F}_R^X, P_R^X)$ and $(\Omega_R^Z, \mathcal{F}_R^Z, P_R^Z)$ such that the randomized probability space is
\begin{align*}
(\tilde{\Omega}, \tilde{\mathcal{F}}, \tilde{P}) = (\Omega \times \Omega_R^X \times \Omega_R^Z, \mathcal{F} \times \mathcal{F}_R^X \times \mathcal{F}_R^Z, P \times P_R^X \times P_R^Z)
\end{align*}
\noindent also satisfies Definition \ref{def:joint_randomization}.

Lastly we prove that joint experimental randomization allows us to observe the distribution double-subscripted potential outcomes. This result extends Theorem \ref{thm:random}. 

\begin{thm}\label{thm:joint_random}
Under joint experimental randomization of $X$ and $Z$,
\begin{align*}
P_{Y_{(X, Z) = (x, z)}} = P_{\tilde{Y}|(\tilde{X} = x, \tilde{Z} = z)}
\end{align*} 
\end{thm}

\begin{proof}
The proof is analogous to that of the proof of Theorem \ref{thm:random}. We simply write out the conditional probability explicitly:
\begin{align*}
P_{\tilde{Y} | (\tilde{X} = x, \tilde{Z} = z)}(y) &\equiv \frac{\tilde{P}(\{\tilde{Y} = y\} \cap \{\tilde{X} = x\} \cap \{\tilde{Z} = z\})}{\tilde{P}(\{\tilde{X} = x\} \cap \{\tilde{Z} = z\})} \\  
&= \frac{\tilde{P}\left(\cup_{(x', z')} \left\{ Y_{(X, Z) = (x', z')}^{-1}(y) \times ( X_R^{-1}(x') \cap Z_R^{-1}(z') ) \right\}  \cap \{\Omega \times X_R^{-1}(x)\} \cap \{\Omega \times Z_R^{-1}(z)\} \right)}{\tilde{P}(\{\Omega \times X_R^{-1}(x)\} \cap \{\Omega \times Z_R^{-1}(z)\}} \\
&= \frac{\tilde{P}\left(\cup_{(x', z')} \left\{ Y_{(X, Z) = (x', z')}^{-1}(y) \times ( X_R^{-1}(x') \cap Z_R^{-1}(z') ) \right\}  \cap \{\Omega \times (X_R^{-1}(x) \cap Z_R^{-1}(z))\} \right)}{\tilde{P}(\{\Omega \times (X_R^{-1}(x) \cap Z_R^{-1}(z))\}} \\ 
&= \frac{\tilde{P}\left( Y_{(X, Z) = (x, z)}^{-1}(y) \times ( X_R^{-1}(x') \cap Z_R^{-1}(z') ) \right)}{\tilde{P}(\{\Omega \times (X_R^{-1}(x) \cap Z_R^{-1}(z))\}} \\    
&= \frac{P(Y_{(X, Z) = (x, z)}^{-1}(y))P_R(X_R^{-1}(x) \cap Z_R^{-1}(z))}{P(\Omega) P_R(X_R^{-1}(x) \cap Z_R^{-1}(z))} \\
&=  P(Y_{(X, Z) = (x, z)}^{-1}(y)) \\
&=  P_{Y_{(X, Z) = (x, z)}}(y) 
\end{align*}
\end{proof}

\subsection{Matching from the perspective of probability spaces}

Sections \ref{sec:randomization} and \ref{sec:joint_randomization} describe how experimental randomization can be used to uncover causal relationships. In many applications, however, experimental randomization is not possible due to practical limitations or ethical concerns. As such, much causal inference literature focuses on methods that do not require experimental randomization. While these methods can be applied to data collected in observational settings, they often require strong assumptions. In the present section, we study one particular method for observational causal inference: an elementary matching method called \textit{\textbf{exact paired matching}}. While the shortcomings of exact paired matching have been previously recognized \cite{rosenbaum_1991, cochran_1965}, the purpose of our discussion is to show how a measure theoretic perspective can provide additional clarity. We discuss matching in this section because it necessitates a minimum of three variables: a treatment ($X$), a response ($Y$), and a matching variable ($Z$).


The strategy of exact paired matching is to subsample the original dataset, retaining only pairs of individuals that i) are identical on covariates $\boldsymbol{Z}$ and ii) differ in their receipt of treatment $X$. The matched dataset then consist of $n_M$ triples of the form
\begin{align*}
(Y^{(i,0)}, Y^{(i,1)}, \boldsymbol{Z}^{(i)})
\end{align*}
\noindent where $Y^{(i,0)}$ is the value of the response for the untreated $(X = 0)$ member of the $i^{\text{th}}$ matched pair (and likewise for $Y^{(i,1)}$). Both individuals take the same value $\boldsymbol{Z}^{(i)}$ for the matching variables. The subsampled dataset is then analyzed as if it were obtained from an experimentally randomized experiment. For instance, the sample AOE may be reported as an estimate of the ACE:
\begin{align*}
\widehat{\text{ACE}}_M = \frac{1}{n_M} \sum_i Y^{(i,1)} - \frac{1}{n_M} \sum_i Y^{(i,0)}
\end{align*}

This procedure is motivated by the intuition that matched pairs of individuals are similar in all respects except treatment; it would then seem that differences in their outcomes are more reasonably attributable to differences in their treatment. $\widehat{\text{ACE}}_M$, which may also be written as $\frac{1}{n_M}\sum_{i}(Y^{(i, 1)} - Y^{(0, i)})$, is then the average of these treatment-attributable differences. This argument seems more tenable the more matching variables are used: matched pairs being increasingly comparable. However, including additional covariates in $\boldsymbol{Z}$ poses several challenges. First, even a modest number of covariates may result an unreasonably large set of distinct values of $\boldsymbol{z}$. Twenty binary covariates, for instance, yields more than a million distinct $\boldsymbol{z}$ combinations. A dataset of modest size may have very few (or potentially zero) available triples $(Y^{(i,0)}, Y^{(i,1)}, Z^{(i)})$. The measure theoretic perspective makes clear additional challenges, as we discuss below. 

Let $\boldsymbol{Z} \equiv (Z_1, Z_2, \hdots, Z_K)$ be a set of matching covariates and let $\boldsymbol{Z}_{1:k}$ denote the subset $(Z_1, Z_2, \hdots, Z_k)$. (Note that the subscripts $Z_i$ in this case denote distinct matching variables, rather than potential outcomes.) Let $Z_j(\Omega)$ denote the image of the covariate $Z_j$ and $\boldsymbol{Z}_{1:k}(\Omega) = Z_1(\Omega) \times Z_2(\Omega) \times \hdots \times Z_k(\Omega)$. Throughout this discussion, we assume for simplicity that each matching variable is finite discrete, so that $\boldsymbol{Z}_{1:k}(\Omega)$ is a finite set. We denote $\mathcal{Z}_{1:k}$ as the subset of $\boldsymbol{Z}_{1:k}(\Omega)$ for which both treatment ($X = 1$) and non-treatment ($X = 0$) occur with positive probability:
\begin{align*}
\mathcal{Z}_{1:k} \equiv \{\boldsymbol{z}_{1:k} \in \boldsymbol{Z}_{1:k}(\Omega) : 0 < P(X^{-1}(1) \cap \boldsymbol{Z}_{1:k}^{-1}(\boldsymbol{z}_{1:k})) < 1 \}
\end{align*} 

\noindent We denote $\Omega^M_{1:k}$ the subset of $\Omega$ on which matched pairs are found with positive probability when $\boldsymbol{Z}_{1:k}$ are used as matching variables. $\Omega^M_{1:k}$ can be expressed as follows:
\begin{align}
\Omega^M_{1:k} \equiv \bigcup_{\boldsymbol{z}_{1:k} \in \mathcal{Z}_{1:k}} \boldsymbol{Z}_{1:k}^{-1}(\boldsymbol{z}_{1:k}) \label{eq:omega_M}
\end{align}

\noindent The following result shows that as $k$ increases, $\Omega^M_{1:k}$ form nested subsets.
\begin{thm}[$\Omega^M_{1:k}$ form nested subsets]\label{thm:omega_M}
Under exact paired matching, 
\begin{align*}
\Omega^M_{1:k} \supseteq \Omega^M_{1:(k+1)}
\end{align*}
\end{thm}
\begin{proof}
If $z_{1:k}^* \in \bar{\mathcal{Z}}_{1:k}$ (the complement of $\mathcal{Z}_{1:k}$), then $(z_{1:k}^*, z_{k + 1}) \in \bar{\mathcal{Z}}_{1:(k + 1)}$ for all $z_{k + 1} \in Z_{k + 1}(\Omega)$. This is because if $z_{1:k}^* \in \bar{\mathcal{Z}}_{1:k}$, then $P(X^{-1}(1) \cap Z_{1:k}^{-1}(z_{1:k}^*)) = 0$ or $P(X^{-1}(0) \cap Z_{1:k}^{-1}(z_{1:k}^*)) = 0$. Suppose, without loss of generality, that $P(X^{-1}(1) \cap Z_{1:k}^{-1}(z_{1:k}^*)) = 0$. Then: 
\begin{align*}
P(X^{-1}(1) \cap Z_{1:k}^{-1}(z_{1:k}^*)) &= P\left( X^{-1}(1) \bigcap \left( \bigcup_{z_{k + 1} \in Z_{k+1}(\Omega)} (Z_{1:k}, Z_{k + 1})^{-1}(z_{1:k}^*, z_{k + 1}) \right) \right) \\
&= P\left( \bigcup_{z_{k + 1} \in Z_{k + 1}(\Omega)} X^{-1}(1) \bigcap (Z_{1:k}, Z_{k + 1})^{-1}(z_{1:k}^*, z_{k + 1}) \right) \\
&= \sum_{z_{k + 1} \in Z_{k + 1}(\Omega)} P(X^{-1}(1) \cap (Z_{1:k}, Z_{k + 1})^{-1}(z_{1:k}^*, z_{k + 1})) \\
&= 0
\end{align*}

\noindent Since the final sum is equal to zero, each term $P(X^{-1}(1) \cap (Z_{1:k}, Z_{k + 1})^{-1}(z_{1:k}^*, z_{k + 1}))$ equals zero. This implies $\bar{\Omega}^M_{1:k} \subseteq \bar{\Omega}^M_{1:(k + 1)}$, which in turn implies $\Omega^M_{1:k} \supseteq \Omega^M_{1:(k + 1)}$, as required. This final implication is a consequence of the equivalence between a conditional and it's corresponding contrapositive.
\end{proof}

Figure \ref{fig:matching} provides intuition for Corollary \ref{thm:omega_M} by visualizing an example of exact paired matching on the square space. The three consecutive panels of Figure \ref{fig:matching} display $\Omega$, $\Omega^M_1$, and $\Omega^M_{1:2}$, respectively, when matching is performed on two finite discrete random variables. The fact that $\Omega \supseteq \Omega^M_1 \supseteq \Omega^M_{1:2}$ can be seen as a consequence of each additional matching variable more finely partitioning $\Omega$ (in other words, $\Omega \subset \Sigma_{Z_1} \subset \Sigma_{(Z_1, Z_2)}$). 

Corollary \ref{thm:omega_M} has two important ramifications for exact paired matching. The first concerns the decreasing size of the matched sample as more matching covariates are included. It has already been mentioned that as more matching covariates are added, the number of distinct combinations may far exceed the size of the sample itself. Corollary \ref{thm:omega_M} makes clear that there is an additional compounding problem. As a simple consequence of Corollary \ref{thm:omega_M}, $P(\Omega^M_{1:k}) \ge P(\Omega^M_{1:(k + 1)})$. Therefore, the space on which viable matched pairs may be found with positive probability reduces in measure as more matching variables are added. The second ramification, closely related to the first, involves an inherent bias. The subset of $\Omega$ ignored entirely under exact paired matching, $\Omega^M_{1:k}$ increases in size as more matching variables are included. If $Y_{X = 0}$ and $Y_{X = 1}$ differ on $\bar{\Omega}^M_{1:k}$ but are identical on $\Omega^M_{1:k}$, any downstream analysis of causality relying the matched sample will miss this causal relationship. Even in the infinite sample limit, this second problem persists. 

Although exact paired matching seeks to emulate the behavior of experimental randomization by creating a dataset in which covariates are independent from treatment, a comparison between Figure \ref{fig:randomization} and Figure \ref{fig:matching} clearly illustrates how different these two procedures are. While experimental randomization allows one to observe samples of the potential outcomes $Y_{X = 0}$ and $Y_{X = 1}$ across the entire original sample space $\Omega$, matching may exclude information from virtually all of $\Omega$. Experimental randomization ensures that the difference in sample means will converge to the true $\text{ACE}$ in the large sample limit. The quantity produced by exact paired matching, $\widehat{\text{ACE}}_M$, has a far more opaque interpretation. Without additional assumptions, it need not converge to any causal quantity of interest. 

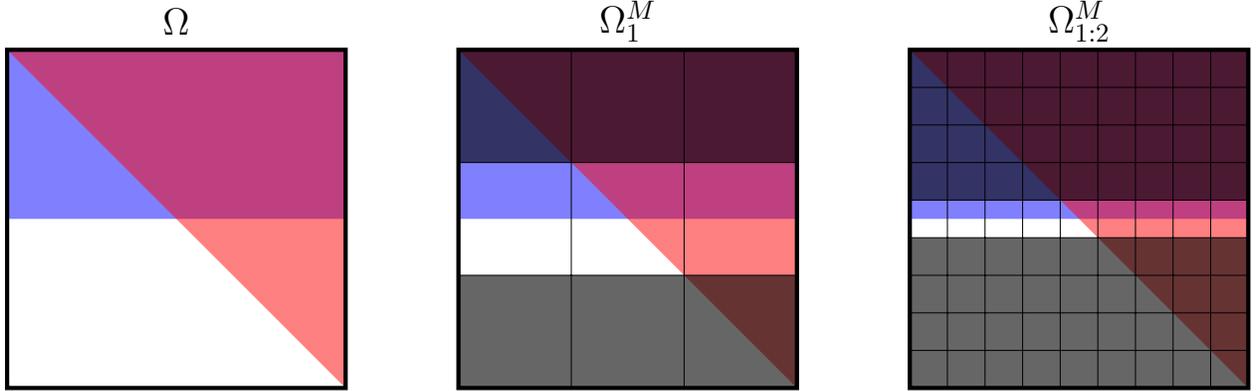
\begin{figure}
\begin{center}
\begin{tikzpicture}[scale = .75]

\coordinate (CO) at (0, 0);
\coordinate (COS) at (8, 0);
\coordinate (COSS) at (16, 0);

\coordinate (nw) at (-3, 3);
\coordinate (ne) at (3, 3);
\coordinate (se) at (3, -3);
\coordinate (sw) at (-3, -3);


\fill[blue, opacity = 0.5] ($(CO) + (nw)$) -- ($(CO) + (ne)$) -- ($(CO) + (3, 0)$) -- ($(CO) + (-3, 0)$) -- cycle;
\fill[red, opacity = 0.5] ($(CO) + (ne)$) -- ($(CO) + (se)$) -- ($(CO) + (nw)$) -- cycle;

\draw[ultra thick] ( $(CO) + (nw)$ ) -- ( $(CO) + (ne)$ ) -- ( $(CO) + (se)$ ) -- ( $(CO) +  (sw)$ ) -- cycle;
\draw[shift = {( $(CO) + (sw)$ )}, step = 10cm, black, very thin] (0, 0) grid ( $(CO) + (3, 3)$ );
\draw ( $(CO) + (0, 3.5)$ ) node {\large $\Omega$};

\fill[blue, opacity = 0.5] ($(COS) + (nw)$) -- ($(COS) + (ne)$) -- ($(COS) + (3, 0)$) -- ($(COS) + (-3, 0)$) -- cycle;
\fill[red, opacity = 0.5] ($(COS) + (ne)$) -- ($(COS) + (se)$) -- ($(COS) + (nw)$) -- cycle;
\fill[black, opacity = 0.6] ($(COS) + (nw)$) -- ($(COS) + (ne)$) -- ($(COS) + (3, 1)$) -- ($(COS) + (-3, 1)$) -- cycle;
\fill[black, opacity = 0.6] ($(COS) + (sw)$) -- ($(COS) + (se)$) -- ($(COS) + (3, -1)$) -- ($(COS) + (-3, -1)$) -- cycle;

\draw[ultra thick] ( $(COS) + (nw)$ ) -- ( $(COS) + (ne)$ ) -- ( $(COS) + (se)$ ) -- ( $(COS) +  (sw)$ ) -- cycle;
\draw[shift = {( $(COS) + (sw)$ )}, step = 2cm, black, very thin] (0, 0) grid ( $(COS) + (3, 3)$ );
\draw ( $(COS) + (0, 3.5)$ ) node {\large $\Omega^M_1$};

\fill[blue, opacity = 0.5] ($(COSS) + (nw)$) -- ($(COSS) + (ne)$) -- ($(COSS) + (3, 0)$) -- ($(COSS) + (-3, 0)$) -- cycle;
\fill[red, opacity = 0.5] ($(COSS) + (ne)$) -- ($(COSS) + (se)$) -- ($(COSS) + (nw)$) -- cycle;
\fill[black, opacity = 0.6] ($(COSS) + (nw)$) -- ($(COSS) + (ne)$) -- ($(COSS) + (3, .3333333)$) -- ($(COSS) + (-3, .3333333)$) -- cycle;
\fill[black, opacity = 0.6] ($(COSS) + (sw)$) -- ($(COSS) + (se)$) -- ($(COSS) + (3, -.3333333)$) -- ($(COSS) + (-3, -.3333333)$) -- cycle;

\draw[ultra thick] ( $(COSS) + (nw)$ ) -- ( $(COSS) + (ne)$ ) -- ( $(COSS) + (se)$ ) -- ( $(COSS) +  (sw)$ ) -- cycle;
\draw[shift = {( $(COSS) + (sw)$ )}, step = .66666666cm, black, very thin] (0, 0) grid ( $(COSS) + (3, 3)$ );
\draw ( $(COSS) + (0, 3.5)$ ) node {\large $\Omega^M_{1:2}$};

\end{tikzpicture}
\end{center}
\caption{(left) The observable random variables $X$ and $Y$ as in Figure \ref{fig:mult}. (middle) The atoms defined by the matching variable $Z_1$, delineated by a $3 \times 3$ grid, partition $\Omega$.  All but $\Omega^M_1$ is occluded. (right) The atoms defined by the matching variables $Z_1$ and $Z_2$, delineated by a $9 \times 9$ grid, partition $\Omega$. All but $\Omega^M_{1:2}$ is occluded.}
\label{fig:matching}
\end{figure}

\section{A general framework for causal systems}\label{sec:generalization}

In Sections \ref{sec:two_vars} and \ref{sec:three_vars}, we carefully examined systems of two or three observable random variables in order to build intuition and demonstrate important features of causality from a measure theoretic perspective. In the present section, we attempt to move beyond these simple (though instructive) systems by providing an axiomatic framework for a general model of causality, which we term an observable causal system (OCS). An OCS builds on the simple models discussed above in two ways. First, we will consider collections of arbitrarily many finite discrete random variables. Second, we allow all random variables to be causal for each other. In providing a more general model, our axiomatic framework has the additional benefit of being amenable to a more formal description of causality. We therefore examine some of the immediate corollaries of our set of causal axioms, and describe how the basic structure of causality emerges as corollaries to the axioms.

\subsection{A comment on notation}
In the present section, we will consider systems of $n$ observable random variables. Since subscripts are reserved for potential outcomes, superscripts will be used to index each of these random variables and should not be confused with exponentiation. Thus, we will denote a set of $n$ random variables as $\{X^{1}, X^{2}, \hdots, X^{n}\}$. 

We will often need to refer to subsets of the random variables $\{X^{1}, X^{2}, \hdots, X^{n}\}$. We denote $\mathcal{S} = \{s_1, s_2, \hdots, s_k\}$ to be an arbitrary subset of the integers $\{1, 2, \hdots, n\}$ and $\bar{\mathcal{S}}$ to be the complement of $\mathcal{S}$ ($\mathcal{S} \cup \bar{\mathcal{S}} = \{1, 2, \hdots, n\}$). We refer to the the multivariate random variable $(X^{s_1}, X^{s_2}, \hdots, X^{s_k})$ as $\boldsymbol{X}^{\mathcal{S}}$. Then,  $\boldsymbol{X}^d_{\boldsymbol{X}^{\mathcal{S}} = \boldsymbol{x}^{\mathcal{S}}}$ denotes the potential outcome 
\begin{align*}
X^{d}_{(X^{s_1}, X^{s_2}, \hdots, X^{s_k}) = (x^{s_1}, x^{s_2}, \hdots, x^{s_k})}
\end{align*} 
\noindent For compactness, sometimes we will use the shorthand $
\boldsymbol{X}^d_{\boldsymbol{x}^{\mathcal{S}}} = \boldsymbol{X}^d_{\boldsymbol{X}^{\mathcal{S}} = \boldsymbol{x}^{\mathcal{S}}}$ when there is no possibility of confusion. As an example, for $n = 4$, $d = 3$, $\mathcal{S} = \{1, 2, 4\}$, and $\boldsymbol{x}^{\mathcal{S}} = (0, 1, 1)$, 
\begin{align*}
X^{d}_{\boldsymbol{x}^{\mathcal{S}}} = X^{d}_{\boldsymbol{X}^{\mathcal{S}} = \boldsymbol{x}^{\mathcal{S}}} = X^{3}_{(X^{1}, X^{2}, X^{4}) = (0, 1, 1)}
\end{align*}
\noindent In words, this potential outcome is ``$X^3$ had $X^1$ been 0, $X^2$ been 1, and $X^4$ been 1."

\subsection{Definition of an observable causal system}\label{sec:axioms}

Let $\boldsymbol{X}$ denote the multivariate random variable $(X^{1}, X^{2}, \hdots, X^{n})$ and $X^{i}(\Omega)$ the image of the random variable $X^{i}$. Assuming that each $X^i$ is  a discrete finite random variable, then $\boldsymbol{X}(\Omega) \equiv X^{1}(\Omega) \times X^{2}(\Omega) \times \hdots \times X^{n}(\Omega)$ is a finite set. We do not consider the setting of continuous random variables in this work. For each $\boldsymbol{x} \in \boldsymbol{X}(\Omega)$, we define the indicator random variable
\begin{align*}
   I_{\boldsymbol{X} = \boldsymbol{x}}(\omega) = \left\{
     \begin{array}{lr}
       1 & : \boldsymbol{X}(\omega) = \boldsymbol{x}\\
       0 & : \boldsymbol{X}(\omega) \ne \boldsymbol{x}
     \end{array}
   \right.
\end{align*}
\begin{defin}[Observable causal system]\label{def:axioms}
A set of random variables $\{ X^{1}, X^{2}, \hdots X^{n} \}$ defined on the probability space $(\Omega, \mathcal{F}, P)$ is an \textbf{\textit{observable causal system}} if the following properties hold:

\begin{enumerate}

\item\label{axiom:complete} \textbf{\textit{Existence of potential outcomes}}: For all $\boldsymbol{x} \in \boldsymbol{X}(\Omega)$, there exists a random variable $X^{i}_{\boldsymbol{X} = \boldsymbol{x}}$ for each $i \in \{1, 2, \hdots, n\}$. These random variables are called \textbf{\textit{complete potential outcomes}}. 

\item\label{axiom:consistency} \textbf{\textit{Observational Consistency}}: The indicators partially determine the complete potential outcomes. Specifically, 
\begin{align*}
\left( I_{\boldsymbol{X} = \boldsymbol{x}}(\omega) = 1 \right) \implies \left( \boldsymbol{X}_{\boldsymbol{X} = \boldsymbol{x}}(\omega) = \boldsymbol{x} \right)
\end{align*}
\noindent We will say that a random variable is \textbf{\textit{identified at $\omega$}} if its value is determined by observational consistency.

\item\label{axiom:partial} \textbf{\textit{Partial consistency}}: For a subset $\mathcal{S} \subseteq \{1, 2, \hdots, n\}$ and $\bar{\mathcal{S}}$ denoting the complement of $\mathcal{S}$, the indicators and potential outcomes may be derived according to the following generalized \textbf{\textit{contraction}} procedures:
\begin{align*}
I_{\boldsymbol{X}^{\mathcal{S}} = \boldsymbol{x}^{\mathcal{S}}}(\omega) &\equiv \sum_{\boldsymbol{x}^{\bar{\mathcal{S}}}} I_{\boldsymbol{X} = \boldsymbol{x}}(\omega) \\
X^{i}_{\boldsymbol{X}^{\mathcal{S}} = \boldsymbol{x}^{\mathcal{S}}}(\omega) &\equiv \sum_{\boldsymbol{x}^{\bar{\mathcal{S}}}} I_{\boldsymbol{X}^{\bar{\mathcal{S}}} = \boldsymbol{x}^{\bar{\mathcal{S}}}}(\omega)X^{i}_{\boldsymbol{X} = \boldsymbol{x}}(\omega) 
\end{align*}
\noindent If $\mathcal{S} \ne \{1, 2, \hdots, n\}$, then $X^{i}_{\boldsymbol{X}^{\mathcal{S}} = \boldsymbol{x}^{\mathcal{S}}}$ is called a \textbf{partial potential outcome}. We say that the set of partial potential outcomes $\{X_{\boldsymbol{X}^{\mathcal{S}} = \boldsymbol{x}^{\mathcal{S}}} \}$ are derived from the complete potential outcomes $\{X_{\boldsymbol{X} = \boldsymbol{x}}\}$ by \textbf{\textit{contracting over $\bar{\mathcal{S}}$}}.

\end{enumerate}
\end{defin}

Definition \ref{def:axioms} formalizes the essential features of causality discussed in Sections \ref{sec:two_vars} and \ref{sec:three_vars}. Axiom \ref{axiom:complete} ensures that all conceivable potential outcomes exist while Axioms \ref{axiom:consistency} and \ref{axiom:partial} constrict how the potential outcomes are related to each other and observable random variables through contraction. As previously mentioned, an OCS generalizes the simple causal systems studied in Sections \ref{sec:two_vars} and \ref{sec:three_vars} in two ways. First, an OCS allows systems of arbitrarily many (rather than just two or three) finite discrete (rather than just binary) random variables. Second, an OCS allows all observable random variables to be causal for one another. 

The first generalization is conceptually straightforward; largely, it is a matter of extending notation. The second generalization is more fundamental. By allowing potentially all random variables to be causal for each other, we need to consider several new types of potential outcomes. In the simple system of two binary observable random variables $X$ and $Y$, we previously only considered the potential outcomes $Y_{X = x}$. Now we also consider the potential outcomes $X_{Y = y}$. Just as $X$ can be causal for $Y$ if $Y_{X = 0} \ne Y_{X = 1}$ on some subset of nonzero measure, so too can $Y$ be causal for $X$ if $X_{Y = 0} \ne X_{Y = 1}$ on some subset of nonzero measure. This allows for the possibility of feedback in an OCS; something which is excluded in DAG models. 

An additional, minor subtlety introduced by Definition \ref{def:axioms} is that of self-referential indices. For example, in an observable causal system of two binary random variables $X$ and $Y$, we assume not only the existence of potential outcomes $Y_{X = 0}$ and $Y_{X = 1}$, but also the complete potential outcomes $Y_{(X, Y) = (0, 0)}$, $Y_{(X, Y) = (0, 1)}$, $Y_{(X, Y) = (1, 0)}$, and $Y_{(X, Y) = (1, 1)}$. Although self-referential indices have no intuitive interpretation, we decide to include them for the sake of notational compactness.  With self-referential indices, the complete potential outcomes for each observable random variable, $\{X^{i}_{\boldsymbol{x}} \}$, take on the same sets of indices for each $i$. Another benefit is that self-referential indices simplify the statement of the observational consistency axiom. Again for two binary random variables, observational consistency simply states that when $X = x$ and $Y = y$, the complete potential outcomes are as expected: $X_{(X, Y) = (x, y)} = x$ and $Y_{(X, Y) = (x, y)} = y$. In practice, we will only consider potential outcomes in which the self-referential index has been contracted out.

\subsection{Generalized definition of causality}
The definition of an OCS provided above suggests the following definition of causality, which generalizes Definition \ref{def:causal_definition}. 

\begin{defin}[Generalized definition of causality]\label{def:general_causality}
A set of variables $\boldsymbol{X}^{\mathcal{S}}$ is causal for a random variable $X^i$ (denoted $\boldsymbol{X}^{\mathcal{S}} \to X^i$) if $X^i_{\boldsymbol{X}^{\mathcal{S}} = \boldsymbol{x}^{\mathcal{S}}} \ne X^i_{\boldsymbol{X}^{\mathcal{S}} = \boldsymbol{\tilde{x}}^{\mathcal{S}}}$ for $\boldsymbol{x}^{\mathcal{S}} \ne \boldsymbol{\tilde{x}}^{\mathcal{S}}$ on a subset $F \in \mathcal{F}$ of nonzero measure.
\end{defin}

\noindent  When $\boldsymbol{X}^\mathcal{S}$ consists of a single binary random variable, Definition \ref{def:general_causality} recapitulates Definition \ref{def:causal_definition}. When $\boldsymbol{X}^{\mathcal{S}}$ consists of multiple random variables, Definition \ref{def:general_causality} describes the joint causality scenario detailed in Section \ref{sec:joint_causality}. One notable feature of Definition \ref{def:general_causality} is that it does not preclude the possibility of feedback. For instance, $X^i \to X^j$ and $X^j \to X^i$ may both be true. As we will see in Section \ref{sec:corollaries} below, a general version of the fundamental problem of causal inference precludes possibility of identifying causal relationships from observable data alone. 

\  

\noindent \textbf{Remark:} Although Definition \ref{def:general_causality} does not specifically exclude the possibility that $i \in \mathcal{S}$, it is only of practical interest when $i \centernot\in \mathcal{S}$. 

\subsection{A detailed example}\label{sec:detailed_example}

In this section we briefly reexamine our example of two binary random variables, introduced in Figure \ref{fig:mult} and discussed in Section \ref{sec:two_vars}, to better understand the structure of an OCS and build intuition for some of the new concepts introduced by Definition \ref{def:axioms}. Axiom \ref{axiom:complete} requires the existence of all of the following complete potential outcomes:
\begin{align*}
\{ &X_{(X, Y) = (0, 0)}, X_{(X, Y) = (0, 1)}, X_{(X, Y) = (1, 0)}, X_{(X, Y) = (1, 1)},\\ &Y_{(X, Y) = (0, 0)}, Y_{(X, Y) = (0, 1)}, Y_{(X, Y) = (1, 0)}, Y_{(X, Y) = (1, 1)} \} 
\end{align*}
\noindent Axiom \ref{axiom:consistency} ensures that the complete potential outcomes are consistent with observations. As a consequence, the complete potential are determined in some regions of $\Omega$ and undetermined in others. We represent the implications of Axiom \ref{axiom:consistency} in Figure \ref{fig:completes}. Axiom \ref{axiom:partial} allows us to derive the partial potential outcome through contraction. In addition to the familiar partial potential outcomes $Y_{X = x}$, we can also derive the partial potential outcomes $Y_{Y = y}$. In contracting over an additional index, the partial potential outcomes are identified over a larger portion of $\Omega$ than the complete potential outcomes. The implications of Axiom \ref{axiom:partial} are represented in Figure \ref{fig:partial}. As a sanity check, one should also verify that the fully contracted random variables, obtained by contracting the partial potential outcomes in Figure \ref{fig:partial} over the single remaining index, recovers the original observable random variables from Figure \ref{fig:mult}. This feature of the OCM will be proved generally in Section \ref{sec:corollaries}. 

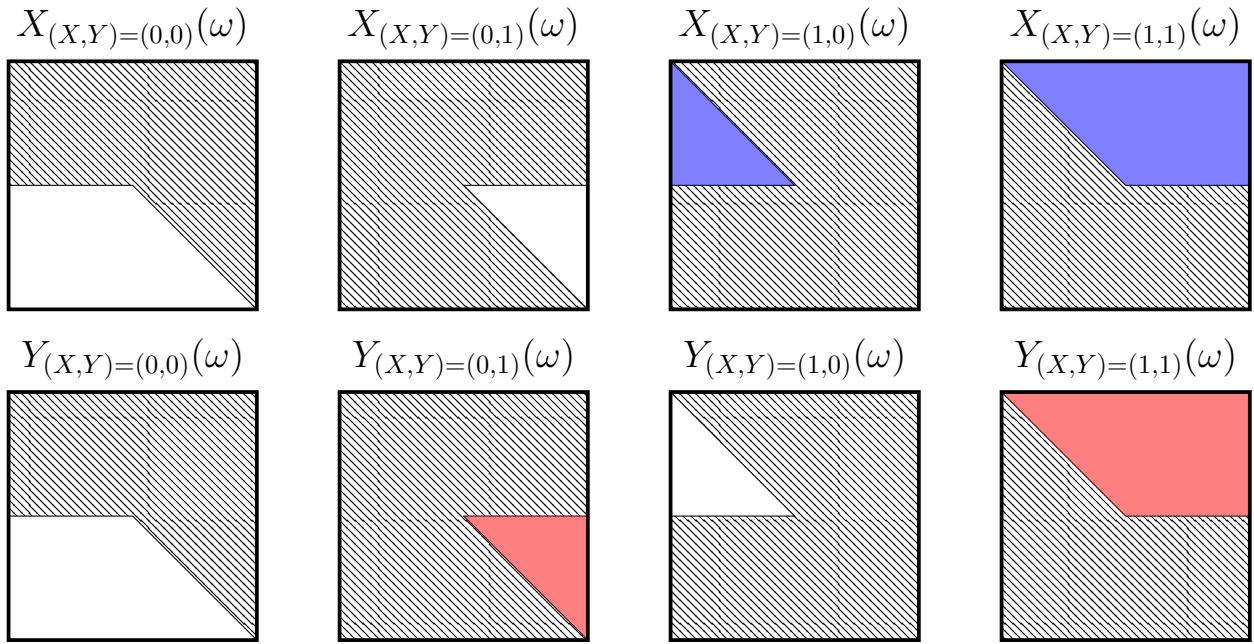
\begin{figure}
\begin{center}
\begin{tikzpicture}[scale = 0.55]

\coordinate (CX00) at (-11, 3);
\coordinate (CX01) at (-3, 3);
\coordinate (CX10) at (5, 3);
\coordinate (CX11) at (13, 3);

\coordinate (CY00) at (-11, -5);
\coordinate (CY01) at (-3, -5);
\coordinate (CY10) at (5, -5);
\coordinate (CY11) at (13, -5);

\coordinate (nw) at (-3, 3);
\coordinate (ne) at (3, 3);
\coordinate (se) at (3, -3);
\coordinate (sw) at (-3, -3);

\draw[pattern=north west lines, pattern color=black] ( $(CX00) + (-3, 0)$ ) -- ( $(CX00) + (nw)$ ) -- ( $(CX00) +  (ne)$ ) -- ( $(CX00) +  (se)$ ) -- ( $(CX00)$ ) -- cycle;
\draw ( $(CX00) + (0, 3.8)$ ) node {\large $X_{(X, Y) = (0, 0)}(\omega)$};
\draw[line width = 0.5mm] ( $(CX00) + (nw)$ ) -- ( $(CX00) + (ne)$ ) -- ( $(CX00) + (se)$ ) -- ( $(CX00) +  (sw)$ ) -- cycle;

\draw[pattern=north west lines, pattern color=black] ( $(CX01) + (nw)$ ) -- ( $(CX01) + (ne)$ ) -- ( $(CX01) + (3, 0)$ ) -- ( $(CX01)$ ) -- ( $(CX01) +  (se)$ ) -- ( $(CX01) +  (sw)$ ) -- cycle;
\draw ( $(CX01) + (0,3.8)$ ) node {\large $X_{(X, Y) = (0, 1)}(\omega)$};
\draw[line width = 0.5mm] ( $(CX01) + (nw)$ ) -- ( $(CX01) + (ne)$ ) -- ( $(CX01) + (se)$ ) -- ( $(CX01) +  (sw)$ ) -- cycle;

\draw[pattern=north west lines, pattern color=black] ( $(CX10) + (nw)$ ) -- ( $(CX10) + (ne)$ ) -- ( $(CX10) + (se)$ ) -- ( $(CX10) +  (sw)$ ) -- ( $(CX10) + (-3, 0)$ ) -- ( $(CX10)$ ) -- cycle;
\fill[blue, opacity = 0.5] ( $(CX10)$ ) -- ( $(CX10) + (-3, 0)$ ) -- ( $(CX10) + (nw)$ ) -- cycle;
\draw ( $(CX10) + (0, 3.8)$ ) node {\large $X_{(X, Y) = (1, 0)}(\omega)$};
\draw[line width = 0.5mm] ( $(CX10) + (nw)$ ) -- ( $(CX10) + (ne)$ ) -- ( $(CX10) + (se)$ ) -- ( $(CX10) +  (sw)$ ) -- cycle;

\draw[pattern=north west lines, pattern color=black] ( $(CX11) + (nw)$ ) -- ( $(CX11)$ ) -- ( $(CX11) + (3, 0)$ ) -- ( $(CX11) + (se)$ ) -- ( $(CX11) +  (sw)$ ) -- cycle;
\fill[blue, opacity = 0.5] ( $(CX11)$ ) -- ( $(CX11) + (3, 0)$ ) -- ( $(CX11) + (ne)$ ) -- ( $(CX11) + (nw)$ ) -- cycle;
\draw ( $(CX11) + (0, 3.8)$ ) node {\large $X_{(X, Y) = (1, 1)}(\omega)$};
\draw[line width = 0.5mm] ( $(CX11) + (nw)$ ) -- ( $(CX11) + (ne)$ ) -- ( $(CX11) + (se)$ ) -- ( $(CX11) +  (sw)$ ) -- cycle;

\draw[pattern=north west lines, pattern color=black] ( $(CY00) + (-3, 0)$ ) -- ( $(CY00) + (nw)$ ) -- ( $(CY00) +  (ne)$ ) -- ( $(CY00) +  (se)$ ) -- ( $(CY00)$ ) -- cycle;
\draw ( $(CY00) + (0, 3.8)$ ) node {\large $Y_{(X, Y) = (0, 0)}(\omega)$};
\draw[line width = 0.5mm] ( $(CY00) + (nw)$ ) -- ( $(CY00) + (ne)$ ) -- ( $(CY00) + (se)$ ) -- ( $(CY00) +  (sw)$ ) -- cycle;

\draw[pattern=north west lines, pattern color=black] ( $(CY01) + (nw)$ ) -- ( $(CY01) + (ne)$ ) -- ( $(CY01) + (3, 0)$ ) -- ( $(CY01)$ ) -- ( $(CY01) +  (se)$ ) -- ( $(CY01) +  (sw)$ ) -- cycle;
\fill[red, opacity = 0.5] ( $(CY01)$ )  -- ( $(CY01) + (3, 0)$ ) -- ( $(CY01) + (se)$ ) -- cycle;
\draw ( $(CY01) + (0, 3.8)$ ) node {\large $Y_{(X, Y) = (0, 1)}(\omega)$};
\draw[line width = 0.5mm] ( $(CY01) + (nw)$ ) -- ( $(CY01) + (ne)$ ) -- ( $(CY01) + (se)$ ) -- ( $(CY01) +  (sw)$ ) -- cycle;

\draw[pattern=north west lines, pattern color=black] ( $(CY10) + (nw)$ ) -- ( $(CY10) + (ne)$ ) -- ( $(CY10) + (se)$ ) -- ( $(CY10) +  (sw)$ ) -- ( $(CY10) + (-3, 0)$ ) -- ( $(CY10)$ ) -- cycle;
\draw ( $(CY10) + (0, 3.8)$ ) node {\large $Y_{(X, Y) = (1, 0)}(\omega)$};
\draw[line width = 0.5mm] ( $(CY10) + (nw)$ ) -- ( $(CY10) + (ne)$ ) -- ( $(CY10) + (se)$ ) -- ( $(CY10) +  (sw)$ ) -- cycle;

\draw[pattern=north west lines, pattern color=black] ( $(CY11) + (nw)$ ) -- ( $(CY11)$ ) -- ( $(CY11) + (3, 0)$ ) -- ( $(CY11) + (se)$ ) -- ( $(CY11) +  (sw)$ ) -- cycle;
\fill[red, opacity = 0.5] ( $(CY11)$ ) -- ( $(CY11) + (3, 0)$ ) -- ( $(CY11) + (ne)$ ) -- ( $(CY11) + (nw)$ ) -- cycle;
\draw ( $(CY11) + (0, 3.8)$ ) node {\large $Y_{(X, Y) = (1, 1)}(\omega)$};
\draw[line width = 0.5mm] ( $(CY11) + (nw)$ ) -- ( $(CY11) + (ne)$ ) -- ( $(CY11) + (se)$ ) -- ( $(CY11) +  (sw)$ ) -- cycle;

\end{tikzpicture}
\end{center}
\caption{Visual representation of Axioms \ref{axiom:complete} and \ref{axiom:consistency}. The complete potential outcomes $X_{(X, Y) = (x, y)}$ and $Y_{(X, Y) = (x, y)}$. Within the shaded (unshaded) regions of $\Omega$, the random variable maps to 1 (0). In the regions filled by diagonal lines, the complete potential outcomes are not determined by observational consistency (i.e., where they are not identified).}\label{fig:completes}
\end{figure}

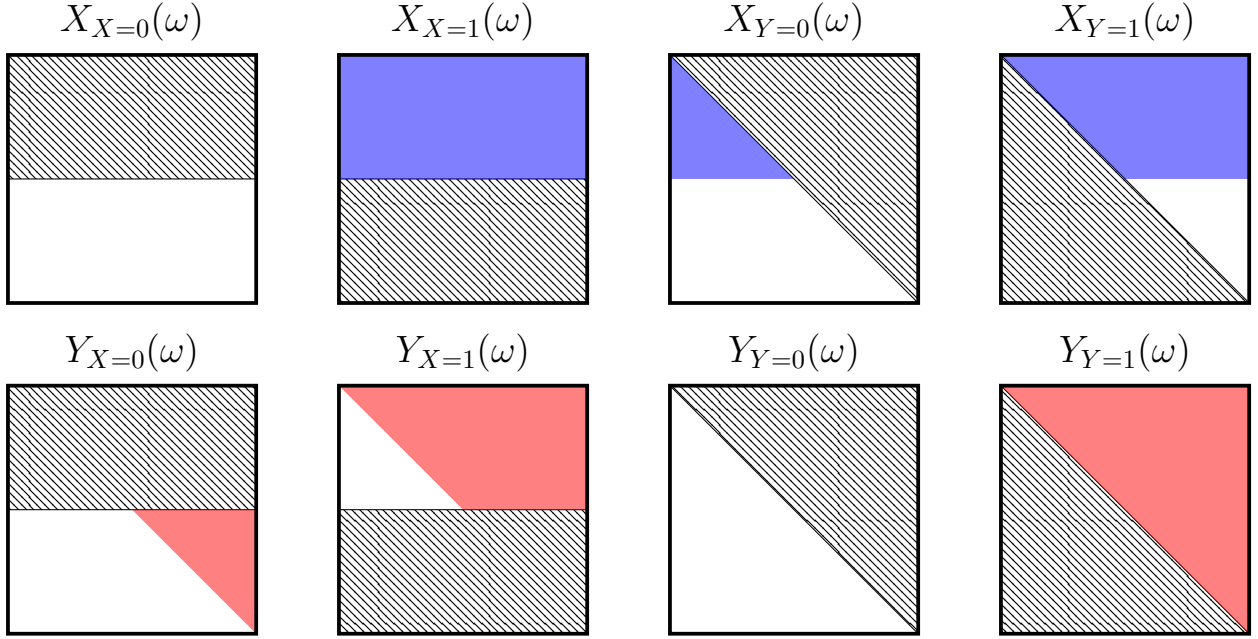
\begin{figure}
\begin{center}
\begin{tikzpicture}[scale = 0.55]

\coordinate (CXX0) at (-11, 3);
\coordinate (CXX1) at (-3, 3);
\coordinate (CXY0) at (5, 3);
\coordinate (CXY1) at (13, 3);

\coordinate (CYX0) at (-11, -5);
\coordinate (CYX1) at (-3, -5);
\coordinate (CYY0) at (5, -5);
\coordinate (CYY1) at (13, -5);

\coordinate (nw) at (-3, 3);
\coordinate (ne) at (3, 3);
\coordinate (se) at (3, -3);
\coordinate (sw) at (-3, -3);

\draw[pattern=north west lines, pattern color=black] ( $(CXX0) + (-3, 0)$ ) -- ( $(CXX0) + (3, 0)$ ) -- ( $(CXX0) +  (ne)$ ) -- ( $(CXX0) +  (nw)$ ) -- cycle;
\draw ( $(CXX0) + (0, 3.8)$ ) node {\large $X_{X = 0}(\omega)$};
\draw[line width = 0.5mm] ( $(CXX0) + (nw)$ ) -- ( $(CXX0) + (ne)$ ) -- ( $(CXX0) + (se)$ ) -- ( $(CXX0) +  (sw)$ ) -- cycle;

\draw[pattern=north west lines, pattern color=black] ( $(CXX1) + (-3, 0)$ ) -- ( $(CXX1) + (3, 0)$ ) -- ( $(CXX1) + (se)$ ) -- ( $(CXX1) +  (sw)$ ) -- cycle;
\fill[blue, opacity = 0.5] ( $(CXX1) + (-3, 0)$ ) -- ( $(CXX1) + (3, 0)$ ) -- ( $(CXX1) + (ne)$ ) -- ( $(CXX1) + (nw)$ ) -- cycle;
\draw ( $(CXX1) + (0,3.8)$ ) node {\large $X_{X = 1}(\omega)$};
\draw[line width = 0.5mm] ( $(CXX1) + (nw)$ ) -- ( $(CXX1) + (ne)$ ) -- ( $(CXX1) + (se)$ ) -- ( $(CXX1) +  (sw)$ ) -- cycle;

\draw[pattern=north west lines, pattern color=black] ( $(CXY0) + (nw)$ ) -- ( $(CXY0) + (ne)$ ) -- ( $(CXY0) + (se)$ ) -- cycle;
\fill[blue, opacity = 0.5] ( $(CXY0)$ ) -- ( $(CXY0) + (-3, 0)$ ) -- ( $(CXY0) + (nw)$ ) -- cycle;
\draw ( $(CXY0) + (0, 3.8)$ ) node {\large $X_{Y = 0}(\omega)$};
\draw[line width = 0.5mm] ( $(CXY0) + (nw)$ ) -- ( $(CXY0) + (ne)$ ) -- ( $(CXY0) + (se)$ ) -- ( $(CXY0) +  (sw)$ ) -- cycle;

\draw[pattern=north west lines, pattern color=black] ( $(CXY1) + (nw)$ ) -- ( $(CXY1) + (sw)$ ) -- ( $(CXY1) + (se)$ ) -- cycle;
\fill[blue, opacity = 0.5] ( $(CXY1)$ ) -- ( $(CXY1) + (3, 0)$ ) -- ( $(CXY1) + (ne)$ ) -- ( $(CXY1) + (nw)$ ) -- cycle;
\draw ( $(CXY1) + (0, 3.8)$ ) node {\large $X_{Y = 1}(\omega)$};
\draw[line width = 0.5mm] ( $(CXY1) + (nw)$ ) -- ( $(CXY1) + (ne)$ ) -- ( $(CXY1) + (se)$ ) -- ( $(CXY1) +  (sw)$ ) -- cycle;

\draw[pattern=north west lines, pattern color=black] ( $(CYX0) + (-3, 0)$ ) -- ( $(CYX0) + (3, 0)$ ) -- ( $(CYX0) +  (ne)$ ) -- ( $(CYX0) +  (nw)$ ) -- cycle;
\fill[red, opacity = 0.5] ( $(CYX0)$ )  -- ( $(CYX0) + (3, 0)$ ) -- ( $(CYX0) + (se)$ ) -- cycle;
\draw ( $(CYX0) + (0, 3.8)$ ) node {\large $Y_{X = 0}(\omega)$};
\draw[line width = 0.5mm] ( $(CYX0) + (nw)$ ) -- ( $(CYX0) + (ne)$ ) -- ( $(CYX0) + (se)$ ) -- ( $(CYX0) +  (sw)$ ) -- cycle;

\draw[pattern=north west lines, pattern color=black] ( $(CYX1) + (-3, 0)$ ) -- ( $(CYX1) + (3, 0)$ ) -- ( $(CYX1) + (se)$ ) -- ( $(CYX1) +  (sw)$ ) -- cycle;
\fill[red, opacity = 0.5] ( $(CYX1)$ )  -- ( $(CYX1) + (3, 0)$ ) -- ( $(CYX1) + (ne)$ ) -- ( $(CYX1) + (nw)$ ) -- cycle;
\draw ( $(CYX1) + (0, 3.8)$ ) node {\large $Y_{X = 1}(\omega)$};
\draw[line width = 0.5mm] ( $(CYX1) + (nw)$ ) -- ( $(CYX1) + (ne)$ ) -- ( $(CYX1) + (se)$ ) -- ( $(CYX1) +  (sw)$ ) -- cycle;

\draw[pattern=north west lines, pattern color=black] ( $(CYY0) + (nw)$ ) -- ( $(CYY0) + (ne)$ ) -- ( $(CYY0) + (se)$ ) -- cycle;
\draw ( $(CYY0) + (0, 3.8)$ ) node {\large $Y_{Y = 0}(\omega)$};
\draw[line width = 0.5mm] ( $(CYY0) + (nw)$ ) -- ( $(CYY0) + (ne)$ ) -- ( $(CYY0) + (se)$ ) -- ( $(CYY0) +  (sw)$ ) -- cycle;

\draw[pattern=north west lines, pattern color=black] ( $(CYY1) + (nw)$ ) -- ( $(CYY1) + (sw)$ ) -- ( $(CYY1) + (se)$ ) -- cycle;
\fill[red, opacity = 0.5] ( $(CYY1) + (nw)$ ) -- ( $(CYY1) + (ne)$ ) -- ( $(CYY1) + (se)$ ) -- cycle;
\draw ( $(CYY1) + (0, 3.8)$ ) node {\large $Y_{Y = 1}(\omega)$};
\draw[line width = 0.5mm] ( $(CYY1) + (nw)$ ) -- ( $(CYY1) + (ne)$ ) -- ( $(CYY1) + (se)$ ) -- ( $(CYY1) +  (sw)$ ) -- cycle;

\end{tikzpicture}
\end{center}
\caption{Visual representation of Axiom \ref{axiom:partial}. The partial potential outcomes $X_{X = x}$, $X_{Y = y}$, $Y_{X = x}$, and $Y_{Y = y}$.Within the shaded (unshaded) regions of $\Omega$, the random variable maps to 1 (0). In the regions filled by diagonal lines, the complete potential outcomes are not identified. }\label{fig:partial}
\end{figure}

\subsection{Corollaries of causal axioms}\label{sec:corollaries}

In this section, we provide some immediate corollaries of the axioms presented in Section \ref{sec:axioms}, thereby making modest steps towards a formal theory of observable causal models. Corollaries \ref{cor:fully_marginalized} and \ref{cor:mult_contractions} describe some basic features of the structure of OCM. More specifically, the describe how observable random variables, partial potential outcomes, and complete potential outcomes relate through the operation of contraction. Corollary \ref{cor:fundamental_problem} is a generalization of the so-called \textit{fundamental problem of causal inference}. 

The operation of contraction allows us to derive partial potential outcomes $\{X_{\boldsymbol{X}^{\mathcal{S}} = \boldsymbol{x}^{\mathcal{S}}}\}$ from the complete potential outcomes $\{X_{\boldsymbol{X} = \boldsymbol{x}}\}$ for any subset $\mathcal{S} \subseteq \{1, 2, \hdots, n \}$. This does not exclude $\mathcal{S} = \emptyset$, in which case we are marginalizing over the full set indices since $\bar{S} = \{1, 2, \hdots, n\}$. The corollary below shows that the observational consistency axiom (Axiom \ref{axiom:consistency}) ensures the expected result.

\begin{cor}[Observables are fully contracted potential outcomes]\label{cor:fully_marginalized}
\begin{align*}
X^{i} = X^{i}_{\emptyset}
\end{align*}
\end{cor}

\begin{proof}
By Axiom \ref{axiom:partial}, we can derive $X^{i}_{\emptyset}$ by fully contracting the complete potential outcomes:
\begin{align*}
X^{i}_{\emptyset}(\omega) &= \sum_{\boldsymbol{x}} I_{\boldsymbol{X} = \boldsymbol{x}}(\omega) X^{i}_{\boldsymbol{X} = \boldsymbol{x}}(\omega) \\
\end{align*}
\noindent All but one of the terms in the above summation are zero. Without loss of generality, let us assume $I_{\boldsymbol{X} = \boldsymbol{x}^*}(\omega) = 1$ with $\boldsymbol{x}^* = (x^*_1, \hdots, x^*_n)$. Then 
\begin{align*}
X^{i}_{\emptyset}(\omega) &= I_{\boldsymbol{x}^*}(\omega)X^{i}_{\boldsymbol{x}^*}(\omega) \\
&= x^*_i
\end{align*}
\noindent by Axiom \ref{axiom:consistency}. By the definition of the indicator random variable, $\boldsymbol{X}(\omega) = \boldsymbol{x}^*$, and so $X^{i}(\omega) = x^*_i$, as required.

\end{proof}

The operation of contraction allows us to consider sub-OCMs: OCMs nested within larger OCMs. For example, consider an OCM on three observable random variables $X$, $Y$, and $Z$, which we denote $\text{OCM}_{XYZ}$. Suppose that data for $Z$ has been discarded, or was never recorded. Any downstream analysis of data from $X$ and $Y$ alone concerns only the sub-OCM on $X$ and $Y$, denoted $\text{OCM}_{XY}$. In $\text{OCM}_{XY}$ the complete potential outcomes $Y_{(X, Y) = (x, y)}$ and $X_{(X, Y) = (x, y)}$ are defined by contracting over $z$ in the original OCM. 

Self-consistency for all such sub-OCMs requires that operation of contraction be well-behaved in various ways. For instance, Corollary \ref{cor:fully_marginalized} shows that $Y$ can be obtained by fully contracting the complete potential outcomes $Y_{(X, Y, Z) = (x, y, z)}$ of $\text{OCM}_{XYZ}$. Since $\text{OCM}_{XY}$ is an OCM on $X$ and $Y$, then fully contracting the complete potential outcomes $Y_{(X, Y) = (x, y)}$ of $\text{OCM}_{XY}$ should also recover $Y$. This requires that first contracting over $z$ and then contracting over $x$ and $y$ is equivalent to contracting over $x$, $y$, and $z$ simultaneously. This and other properties of contraction are summarized in Corollary \ref{cor:mult_contractions} below.

\begin{cor}[Composition of contractions]\label{cor:mult_contractions}
Let us define $M_{\bar{\mathcal{S}}}$ as the operation of contraction over the set $\bar{\mathcal{S}}$. Then 
\begin{align*}
M_{\bar{\mathcal{S}}_1} \circ M_{\bar{\mathcal{S}}_2} \circ \hdots \circ M_{\bar{\mathcal{S}}_n} = M_{\bar{\mathcal{S}}_1 \cup \bar{\mathcal{S}}_2 \cup \hdots \cup \bar{\mathcal{S}}_n} 
\end{align*}
\end{cor}

\begin{proof}
In Appendix \ref{app:mult_contractions}, we show that $M_{\bar{\mathcal{S}}_1} \circ M_{\bar{\mathcal{S}}_2} = M_{\bar{\mathcal{S}}_1 \cup \bar{\mathcal{S}}_2}$. The full result will follow by simply applying it to an arbitrary sequence of contractions, as follows:
\begin{align*}
M_{\bar{\mathcal{S}}_1} \circ M_{\bar{\mathcal{S}}_2} \circ \hdots \circ M_{\bar{\mathcal{S}}_n} &= M_{\bar{\mathcal{S}}_1} \circ M_{\bar{\mathcal{S}}_2} \circ \hdots \circ M_{\bar{\mathcal{S}}_{n - 1} \cup \bar{\mathcal{S}}_n} \\
&= M_{\bar{\mathcal{S}}_1} \circ M_{\bar{\mathcal{S}}_2} \circ \hdots \circ M_{\bar{\mathcal{S}}_{n - 2} \cup \bar{\mathcal{S}}_{n - 1} \cup \bar{\mathcal{S}}_n} \\
& \vdots \\
&= M_{\bar{\mathcal{S}}_1 \cup \bar{\mathcal{S}}_{2} \cup \hdots \cup \bar{\mathcal{S}}_n}
\end{align*}
\end{proof}
\noindent A useful observation from Corollary \ref{cor:mult_contractions} is that contraction is commutative and associative, since these are properties of the union operator.

Finally, we show that a very general version of the fundamental problem of causal inference, discussed in Section \ref{sec:potential_outcomes} emerges from our axiomatic framework. In the simplest setting, where we are trying to understand the effect of a single binary random variable $X$ on another random variable $Y$, the fundamental problem of causal inference states that the potential outcomes $Y_{X = 0}$ and $Y_{X = 1}$ are never simultaneously observable. Corollary \ref{cor:fundamental_problem} below shows that in an OCS, no two potential outcomes (complete or partial) are simultaneously observable.

\begin{cor}[Generalized statement of the fundamental problem of causal inference]\label{cor:fundamental_problem}
Let $\mathcal{X}^{i}_{\boldsymbol{x}^{\mathcal{S}}}$ be defined as follows:
\begin{align*}
\mathcal{X}^{i}_{\boldsymbol{x}^{\mathcal{S}}} \equiv \{ \omega : X^{i}_{\boldsymbol{x}^{\mathcal{S}}}(\omega) \text{ is identified} \}
\end{align*} 
\noindent If $\boldsymbol{x}^{\mathcal{S}} \ne \boldsymbol{\tilde{x}}^{\mathcal{S}}$, then
\begin{align*}
\mathcal{X}^{i}_{\boldsymbol{x}^{\mathcal{S}}} \cap \mathcal{X}^{i}_{\boldsymbol{\tilde{x}}^{\mathcal{S}}} = \emptyset
\end{align*} 
\end{cor}
\begin{proof}
First we note that $\mathcal{X}_{\boldsymbol{x}^{\mathcal{S}}} = \mathcal{I}_{\boldsymbol{x}^{\mathcal{S}}}$ from Lemma \ref{lemma:identification} of Appendix \ref{app:fundamental_problem}. Then we use Lemma \ref{lemma:partition} from Appendix \ref{app:fundamental_problem} to note that $\mathcal{I}_{\boldsymbol{x}^{\mathcal{S}}} \cap \mathcal{I}_{\boldsymbol{\tilde{x}}^{\mathcal{S}}} = \emptyset$
\end{proof}

Analogous to our observations from Section \ref{sec:causal_effects}, there is a natural conflict between the definition of causality provided by Definition \ref{def:general_causality} and the fundamental problem of causal inference stated above. In particular, Definition \ref{def:general_causality} states that $\boldsymbol{X}^{\mathcal{S}} \to X^i$ if $X^i_{\boldsymbol{X}^{\mathcal{S}} = \boldsymbol{x}^{\mathcal{S}}} \ne X^i_{\boldsymbol{X}^{\mathcal{S}} = \boldsymbol{\tilde{x}}^{\mathcal{S}}}$ for $\boldsymbol{x}^{\mathcal{S}} \ne \tilde{\boldsymbol{x}}^{\mathcal{S}}$ on some set of positive measure. Corollary \ref{cor:fundamental_problem}, however, establishes that $X^i_{\boldsymbol{X}^{\mathcal{S}} = \boldsymbol{x}^{\mathcal{S}}}$ and  $X^i_{\boldsymbol{X}^{\mathcal{S}} = \boldsymbol{\tilde{x}}^{\mathcal{S}}}$ are never simultaneously identified. In the absence of randomization or additional assumptions, Corollary \ref{cor:fundamental_problem} reiterates that causal effects generally cannot be determined from observable data.

\section{Discussion}

In this work, we have described the interface between causal inference and classical probability and made initial steps towards developing a mathematically axiomatized theory of observable causal models. Our discussion has been centered around a careful examination of simple systems, each highlighting the utility of a measure theoretic perspective on different aspects of causal inference. 

There are many important questions that we have only begun to consider, and hope that this work will initiate deeper inquiry into the relationship between causality and probability. In extending the mathematical development of observable causal systems, an essential future step will be the inclusion of continuous random variables. Additionally, the causal concepts that we have explored in this work---including causal effects, causal interactions, randomization, and matching---are by no means exhaustive. We anticipate that a measure theoretic description of many more causal concepts will also be useful.

Throughout this work, measure theory has provided clarity and definitions to abstract concepts. As such, we have only needed the most elementary tools from measure theory. We believe that measure theory can also play a more constructive role in the development of causal inference. Measure theoretic machinery has enabled many important advances, otherwise intractable, in the development of probability and statistics. We believe that this will also be true in the future development of causal inference. 

\section*{Acknowledgements}

This research was supported in part by NIH grant HG006448.

\clearpage
\bibliography{refs}
\bibliographystyle{unsrt}

\clearpage
\appendix
\section{Review of classical probability theory}\label{app:detailed_background}

\subsection{The probability space}
\begin{defin}[probability space]\label{def:probability_space}
A \textbf{\textit{probability space}}, denoted $(\Omega, \mathcal{F}, P)$, consists of three objects: 

\begin{enumerate}[(i)]
\item $\Omega$: A set called the \textbf{\textit{sample space}}. 
\item $\mathcal{F}$: A set of subsets of $\Omega$. $\mathcal{F}$ must contain $\Omega$ and be closed under complementations and countable unions (i.e., $\mathcal{F}$ is a $\sigma$-algebra). Elements $F \in \mathcal{F}$ are called \textbf{\textit{events}}.
\item $P$: A real-valued function defined on events $F \in \mathcal{F}$. $P$ must have three properties: a) it must be nonnegative b) $P(\Omega) = 1$ and c) for any countable sequence of mutually exclusive events, $P\left(\cup_{i = 1} F_i \right) = \sum_{i} P(F_i)$. $P$ is called the \textbf{\textit{probability measure}}.
\end{enumerate}
\end{defin}

In the measure theoretic framework, randomness originates from the selection of elements $\omega$ (called \textbf{\textit{random outcomes}}) from a set $\Omega$ (rcalled the \textbf{\textit{sample space}}). The probability with which different outcomes $\omega \in \Omega$ are selected is encoded by the \textbf{\textit{probability measure}} $P$. In some simple settings, $P$ will explicitly define the probability of each random outcome. For example, when $\Omega = \{\omega_1, \omega_2, \hdots, \omega_n\}$ is a finite set, $P(\omega_i)$ defines the probability that random outcome $\omega_i$ is selected. More generally however, $P$ is defined on subsets of $\Omega$, the \textbf{\textit{events}} $F \in \mathcal{F}$. Intuitively, the probability that the selected random outcome $\omega$ belongs to a particular event \textbf{\textit{event}} $F \in \mathcal{F}$ is $P(F)$ \cite{williams}. An event $F$ is said to have \textit{\textbf{measure}} $P(F)$. 

\subsection{Random variables, distributions, and expectations}
Typically, the probability space $(\Omega, \mathcal{F}, P)$ is not directly observable. Instead, we observable random variables. A random variable $X$ is a \textbf{\textit{$\mathcal{F}$-measurable function}}, meaning that it has the following properties:
\begin{align*}
X &: (\omega \in \Omega) \to \mR \\ 
X^{-1} &: (B \in \mathcal{B}) \to (F \in \mathcal{F})
\end{align*}
\noindent In other words, random variables map elements of $\omega \in \Omega$ to $\mR$ in such a way that the pre-image of sets $B \in \mathcal{B}$ are events $F \in \mathcal{F}$. 

A random variable $X$ and a probability measure $P$ define the \textbf{\textit{probability law}} $P_X$ of $X$ in the following way:  
\begin{align*}
P_X(B) &\equiv P \circ X^{-1} (B) \\
&= P(\{\omega : X(\omega) \in B \})
\end{align*}
\noindent The probability law precisely characterizes our uncertainty in $X$. For finite discrete random variables, to which we limit ourselves in this work, the picture is simple. Denoting $\{x_1, \hdots, x_k\}$ as the image of a random variable $X$, the probability law is characterized by the events $\{F_{x_1} = X^{-1}(x_1),\hdots, F_{x_k} = X^{-1}(x_k)$\}, which partition $\Omega$.

In full generality, the notion of the \textbf{\textit{expectation}} of a random variable is an involved topic within the measure theoretic framework. However, for our purposes, the following simple definition for a finite discrete random variable $X$ is sufficient:
\begin{equation}\label{eq:expectation}
\E[X] \equiv \sum_{i = 1}^k x_i P_X(x_i) = \sum_{i = 1}^k x_i P(F_{x_i})
\end{equation}
\subsection{Multiple random variables}\label{sec:mult}


If $Y$ is another random variable on the same probability space $(\Omega, \mathcal{F}, P)$, the multivariate random variable $(X,Y)$ is constructed in the natural way:
\begin{displaymath}
(X,Y)(\omega) = (X(\omega), Y(\omega)) \in \mR^2
\end{displaymath}
\noindent We can define the joint probability law $P_{X,Y}$ analogously to the univariate case:
\begin{align*}
P_{X,Y}(B_2) &\equiv P \circ (X,Y)^{-1} (B_2) \\ 
& = P(\{\omega : (X,Y) (\omega) \in B_2 \})
\end{align*}
\noindent where the $B_2$ (an open disc, for example) is an element of $\mathcal{B}_2$, the Borel $\sigma-$algebra on $\mR^2$. This joint distribution $P_{X,Y}$ completely determines the marginal distributions $P_X$ and $P_Y$. Specifically, we have:
\begin{align*}
P_X(B_X)  &\equiv P_{X, Y}(B_X \times \mathbb{R})  \\
P_Y(B_Y)  &\equiv P_{X, Y}(\mathbb{R} \times B_Y)
\end{align*}
\noindent Where $\times$ denotes the Cartesian product. Generalizing the above concepts to any number of random variables is straightforward.

The random variable $X$ can provide information about the random variable $Y$ through the conditional distribution of $Y$ given $X$. Assuming $P_X(x) \ne 0$, and that both $X$ and $Y$ are finite discrete random variables, the conditional probability law for $Y$ given $X = x$ is defined as follows: 
\begin{align}\label{eq:conditional_law}
P_{Y|X = x}(y) &\equiv \frac{P_{X,Y}(x, y)}{P_X(x)} \nonumber\\
&= \frac{P(F_y \cap F_x)}{P(F_x)}
\end{align}
\noindent From Equation \ref{eq:conditional_law}, we see that the conditional probability law $P_{Y|X=x}$ depends only on the behavior of the random variables $X$ and $Y$ within the subset $F_x$ of $\Omega$. When the random variable $Y$ behaves differently on different subsets $F_x$ of $\Omega$, then the value of $X$ is informative about $Y$. When this is the case, the conditional probability laws $P_{Y|X = x}$ are different for different values of $x$, and the random variables $X$ and $Y$ are called \textbf{\textit{dependent}}. Otherwise, when $Y$ behaves identically on every subset $F_x$ of $\Omega$, then the conditional probability laws $P_{Y|X = x}$ are identical for every value of $x$ and $X$ and $Y$ are called \textbf{\textit{independent}}.

\subsection{Product spaces}\label{sec:product_space}

The previous section defines independence through conditional distributions: $X$ and $Y$ are independent if the conditional probability laws $P_{Y|X=x}$ are the same for every value of $x$. The present section discusses how to construct a probability space, called the product space, on which random variables are independent by design. This construction will be useful when we think about randomized experiments in Section \ref{sec:randomization}.

Suppose we have two separate probability spaces $(\Omega_1, \mathcal{F}_1, P_1)$ and $(\Omega_2, \mathcal{F}_2, P_2)$. The product space provides a prescription for combining the two probability spaces into a single one. The three components of this \textbf{\textit{product space}} $(\Omega, \mathcal{F}, P)$ are built from the components of the original probability spaces in the following natural way:

\begin{enumerate}[(i)]
\item $\Omega$: The sample space $\Omega$ is simply the Cartesian product of the original two sample spaces:
\begin{align*}
\Omega \equiv \Omega_1 \times \Omega_2 = \{(\omega_1, \omega_2) : \omega_1 \in \Omega_1, \omega_2 \in \Omega_2 \}
\end{align*}
\item $\mathcal{F}$: A product event $F = F_1 \times F_2$ is defined as follows: $F \equiv \{(\omega_1, \omega_2) : \omega_1 \in F_1, \omega_2 \in F_2\}$. The product $\sigma$-algebra $\mathcal{F}$, is defined as the smallest $\sigma$-algebra containing all of the product events $\mathcal{F}_1 \times \mathcal{F}_2 = \{F_1 \times F_2: F_1 \in \mathcal{F}_1, F_2 \in \mathcal{F}_2 \}$.  

\item $P$: The product probability measure $P$ is generated by the rule
\begin{equation}\label{eq:product}
P(F_1 \times F_2) = P_1(F_1)P_2(F_2) 
\end{equation}
\noindent This measure is called the \textbf{\textit{product measure}}, and is denoted $P = P_1 \times P_2$. 
\end{enumerate}

If a random variable $X$ is defined on $(\Omega_1, \mathcal{F}_1, P_1)$ and another random variable $Y$ is defined on $(\Omega_2, \mathcal{F}_2, P_2)$, then $X$ and $Y$ will be independent random variables on the product space
\begin{align*}
(\Omega, \mathcal{F}, P) \equiv (\Omega_1 \times \Omega_2, \mathcal{F}_1 \times \mathcal{F}_2, P_1 \times P_2)
\end{align*}
\noindent The product space's ability to induce independence between random variables will be a useful feature in this work. In particular, when we define a randomized experiment in Section \ref{sec:randomization}, the ``treatment" $X$ a ``outcome" $Y$ will live on a product space. 

\section{Proof of Corollary \ref{cor:mult_contractions}}\label{app:mult_contractions}         
\begin{lem}[Composition of two contractions]\label{lem:mult_contractions}
\begin{align*}
M_{\bar{\mathcal{S}}_1} \circ M_{\bar{\mathcal{S}}_2} = M_{\bar{\mathcal{S}}_1 \cup \bar{\mathcal{S}}_2}
\end{align*}
\end{lem}

\begin{proof}
We formally define $M_{\bar{\mathcal{S}}}$ as a mapping between sets of random variables:

\begin{align*}
M_{\bar{\mathcal{S}}}: \{ X^i_{\boldsymbol{X}^{\mathcal{S}'} = \boldsymbol{x}^{\mathcal{S}'}} \}_{\boldsymbol{x}^{\mathcal{S}'}} \to \{ X^i_{\boldsymbol{X}^{\mathcal{S}^*} = \boldsymbol{x}^{\mathcal{S}^*}} \}_{\boldsymbol{x}^{\mathcal{S}^*}}
\end{align*}

\noindent where $\{ X^i_{\boldsymbol{X}^{\mathcal{S}} = \boldsymbol{x}^{\mathcal{S}}} \}_{\boldsymbol{x}^{\mathcal{S}}}$ denotes the set of potential outcomes $X^i_{\boldsymbol{X}^{\mathcal{S}} = \boldsymbol{x}^{\mathcal{S}}}$ for all values of $\boldsymbol{x}^{\mathcal{S}}$, $\mathcal{S}^* = \mathcal{S}' \setminus \bar{\mathcal{S}}$, and 

\begin{align*}
X^i_{\boldsymbol{X}^{\mathcal{S}^*} = \boldsymbol{x}^{\mathcal{S}^*}} = \sum_{\boldsymbol{x}^{\bar{\mathcal{S}} \cap \mathcal{S}'}} I_{\boldsymbol{X}^{\bar{\mathcal{S}} \cap \mathcal{S}'} = \boldsymbol{x}^{\bar{\mathcal{S}} \cap \mathcal{S}'}}  X^i_{(\boldsymbol{X}^{\mathcal{S}'\setminus \bar{\mathcal{S}}}, \boldsymbol{X}^{\mathcal{S}'\cap \bar{\mathcal{S}}}) = (\boldsymbol{x}^{\mathcal{S}'\setminus \bar{\mathcal{S}}}, \boldsymbol{x}^{\mathcal{S}'\cap \bar{\mathcal{S}}})}
\end{align*}

\noindent In words, the operator $M_{\bar{\mathcal{S}}}$ returns a set of potential outcomes where the common indices $\bar{\mathcal{S}} \cap \mathcal{S}'$ are removed by the laws of contraction. 

Slightly abusing notation, let we will abbreviate $X^{i}_{\boldsymbol{X} = \boldsymbol{x}}$ by $X^{i}_{\boldsymbol{x}}$. Similarly, we will also write $X^{i}_{(\boldsymbol{X}^{\mathcal{S}_a}, \boldsymbol{X}^{\mathcal{S}_b}) = (\boldsymbol{x}^{\mathcal{S}_a}, \boldsymbol{x}^{\mathcal{S}_b})}$ by $X^{i}_{\boldsymbol{x}^{\mathcal{S}_a} \boldsymbol{x}^{\mathcal{S}_b}}$, where $\boldsymbol{x}^{\mathcal{S}_a}$ specifies the elements of $\boldsymbol{x}$ indexed by $\mathcal{S}_a$ for any index set $\mathcal{S}_a$ (and likewise for $\boldsymbol{x}^{\mathcal{S}_b}$). 

We with a subset of potential outcomes 
\begin{align*}
\left\{ X^i_{\boldsymbol{x}^{\mathcal{A}}} \right\}_{\boldsymbol{x}^{\mathcal{A}}}
\end{align*}
\noindent for any index set $\mathcal{A} \subset \{1, 2, \hdots, n\}$. We note that we can always have the decomposition 
\begin{align*}
\mathcal{A} = (\mathcal{A} \setminus \bar{\mathcal{S}}_2) \cup (\mathcal{A} \cap \bar{\mathcal{S}}_2)
\end{align*} 
\noindent Therefore, the set:
\begin{align*}
M_{\bar{S}_2}\left( \left\{ X^i_{\boldsymbol{x}^{\mathcal{A}}} \right\}_{\boldsymbol{x}^{\mathcal{A}}} \right) 
\end{align*}
\noindent will have elements
\begin{align*}
X^i_{\boldsymbol{x}^{\mathcal{S}^{*}}} &= \sum_{\boldsymbol{x}^{\bar{\mathcal{S}}_2 \cap \mathcal{A}}} I_{\boldsymbol{x}^{\bar{\mathcal{S}}_2 \cap \mathcal{A}}} X^i_{\boldsymbol{x}^{\mathcal{A} \setminus \bar{\mathcal{S}}_2} \boldsymbol{x}^{\bar{\mathcal{S}}_2 \cap \mathcal{A}}}
\end{align*}
\noindent Now applying $M_{\bar{S}_1}$ to the set $\left\{X^i_{\boldsymbol{x}^{\mathcal{A} \setminus \bar{\mathcal{S}}_2}} \right\}_{\boldsymbol{x}^{\mathcal{A} \setminus \bar{\mathcal{S}}_2}}
$, we have each element 
\begin{align*}
X^i_{\boldsymbol{x}^{\mathcal{S}^{**}}} &= \sum_{\boldsymbol{x}^{\bar{\mathcal{S}}_1 \cap (\mathcal{A} \setminus \bar{\mathcal{S}}_2)}} I_{\boldsymbol{x}^{\bar{\mathcal{S}}_1 \cap (\mathcal{A} \setminus \bar{\mathcal{S}}_2)}} X^i_{\boldsymbol{x}^{\mathcal{S}^{*}}} \\
&= \sum_{\boldsymbol{x}^{\bar{\mathcal{S}}_1 \cap (\mathcal{A} \setminus \bar{\mathcal{S}}_2)}} I_{\boldsymbol{x}^{\bar{\mathcal{S}}_1 \cap (\mathcal{A} \setminus \bar{\mathcal{S}}_2)}} \sum_{\boldsymbol{x}^{\bar{\mathcal{S}}_2 \cap \mathcal{A}}} I_{\boldsymbol{x}^{\bar{\mathcal{S}}_2 \cap \mathcal{A}}} X^i_{\boldsymbol{x}^{\mathcal{A} \setminus \bar{\mathcal{S}}_2} \boldsymbol{x}^{\bar{\mathcal{S}}_2 \cap \mathcal{A}}} \\
&= \sum_{\boldsymbol{x}^{\bar{\mathcal{S}}_1 \cap (\mathcal{A} \setminus \bar{\mathcal{S}}_2)}} I_{\boldsymbol{x}^{\bar{\mathcal{S}}_1 \cap (\mathcal{A} \setminus \bar{\mathcal{S}}_2)}} \sum_{\boldsymbol{x}^{\bar{\mathcal{S}}_2 \cap \mathcal{A}}} I_{\boldsymbol{x}^{\bar{\mathcal{S}}_2 \cap \mathcal{A}}} X^i_{\boldsymbol{x}^{(\mathcal{A} \setminus \bar{\mathcal{S}}_2) \setminus \bar{\mathcal{S}}_1} \boldsymbol{x}^{(\mathcal{A} \setminus \bar{\mathcal{S}}_2) \cap \bar{\mathcal{S}}_1} \boldsymbol{x}^{\bar{\mathcal{S}}_2 \cap \mathcal{A}}} \\
&= \sum_{\boldsymbol{x}^{(\bar{\mathcal{S}}_1 \cap (\mathcal{A} \setminus \bar{\mathcal{S}}_2)) \cup (\bar{\mathcal{S}}_2 \cap \mathcal{A})}} I_{\boldsymbol{x}^{(\bar{\mathcal{S}}_1 \cap (\mathcal{A} \setminus \bar{\mathcal{S}}_2)) \cup (\bar{\mathcal{S}}_2 \cap \mathcal{A})}}  X^i_{\boldsymbol{x}^{(\mathcal{A} \setminus \bar{\mathcal{S}}_2) \setminus \bar{\mathcal{S}}_1} \boldsymbol{x}^{(\bar{\mathcal{S}}_1 \cap (\mathcal{A} \setminus \bar{\mathcal{S}}_2)) \cup (\bar{\mathcal{S}}_2 \cap \mathcal{A})}} \\
&= \sum_{\boldsymbol{x}^{\mathcal{A}\cap (\bar{\mathcal{S}}_1 \cup \bar{\mathcal{S}}_1)}} I_{\boldsymbol{x}^{\mathcal{A}\cap (\bar{\mathcal{S}}_1 \cup \bar{\mathcal{S}}_1)}}  X^i_{\boldsymbol{x}^{(\mathcal{A} \setminus \bar{\mathcal{S}}_2) \setminus \bar{\mathcal{S}}_1} \boldsymbol{x}^{\mathcal{A}\cap (\bar{\mathcal{S}}_1 \cup \bar{\mathcal{S}}_1)}} \\
&= \sum_{\boldsymbol{x}^{\mathcal{A}\cap (\bar{\mathcal{S}}_1 \cup \bar{\mathcal{S}}_1)}} I_{\boldsymbol{x}^{\mathcal{A}\cap (\bar{\mathcal{S}}_1 \cup \bar{\mathcal{S}}_1)}}  X^i_{\boldsymbol{x}^{\mathcal{A} \setminus (\bar{\mathcal{S}}_2 \cup \bar{\mathcal{S}}_1)} \boldsymbol{x}^{\mathcal{A}\cap (\bar{\mathcal{S}}_1 \cup \bar{\mathcal{S}}_1)}} \\
\end{align*}
\noindent Each step is just a tedious exercise of keeping track of indices. The fifth line follows from the fourth by noticing that $(\bar{\mathcal{S}}_1 \cap (\mathcal{A} \setminus \bar{\mathcal{S}}_2)) \cup (\bar{\mathcal{S}}_2 \cap \mathcal{A}) = \mathcal{A}\cap (\bar{\mathcal{S}}_1 \cup \bar{\mathcal{S}}_1)$. The sixth line follows from the fifth by noticing that $(\mathcal{A} \setminus \bar{\mathcal{S}}_2) \setminus \bar{\mathcal{S}}_1 = \mathcal{A} \setminus (\bar{\mathcal{S}}_2 \cup \bar{\mathcal{S}}_1)$. Finally, we notice that the final line is an element of 
\begin{align*}
M_{\bar{\mathcal{S}}_1 \cup \bar{\mathcal{S}}_1} \left( \left\{ X^i_{\boldsymbol{x}^{\mathcal{A}}} \right\}_{\boldsymbol{x}^{\mathcal{A}}} \right)
\end{align*}
\noindent as required.

%
\end{proof}

\section{Proof of Corollary \ref{cor:fundamental_problem}}\label{app:fundamental_problem}
\begin{lem}[indicators partition the sample space]\label{lemma:partition}
Let $\mathcal{I}_{\boldsymbol{x}^{\mathcal{S}}}$ be the subset of $\Omega$ that is mapped to 1 by the indicator $I_{\boldsymbol{x}^{\mathcal{S}}}$:
\begin{align*}
\mathcal{I}_{\boldsymbol{x}^{\mathcal{S}}} = \{\omega: I_{\boldsymbol{x}^{\mathcal{S}}}(\omega) = 1 \}
\end{align*}
\noindent Then 

\begin{enumerate}
\item $\cup_{\boldsymbol{x}^{\mathcal{S}}} \mathcal{I}_{\boldsymbol{x}^{\mathcal{S}}} = \Omega$
\item $\mathcal{I}_{\boldsymbol{x}^{\mathcal{S}}} \cap \mathcal{I}_{\boldsymbol{\tilde{x}}^{\mathcal{S}}} = \emptyset $ for all $\boldsymbol{x}^{\mathcal{S}} \ne \boldsymbol{\tilde{x}}^{\mathcal{S}}$
\end{enumerate}
\end{lem}

\begin{proof}
For $\mathcal{S} = (1, 2, \hdots, n)$ (i.e., when we are considering the statement applied to the complete potential outcomes) both 1 and 2 follow simply from the properties of indicator random variables. When $S \subset \{1, 2, \hdots, n\}$, we note that since $I_{\boldsymbol{x}^{\mathcal{S}}} =\sum_{\boldsymbol{x}^{\bar{\mathcal{S}}}} I_{\boldsymbol{x}}$, we have that:
\begin{align*}
\mathcal{I}_{\boldsymbol{x}^{\mathcal{S}}} = \cup_{\boldsymbol{x}_{\bar{\mathcal{S}}}} \mathcal{I}_{\boldsymbol{x}}
\end{align*}
\noindent Part 1 is now clear because

\begin{align*}
\cup_{\boldsymbol{x}^{\mathcal{S}}} \mathcal{I}_{\boldsymbol{x}^{\mathcal{S}}} &= \cup_{\boldsymbol{x}^{\mathcal{S}}}\cup_{\boldsymbol{x}_{\bar{\mathcal{S}}}}\mathcal{I}_{\boldsymbol{x}} \\
&= \cup_{\boldsymbol{x}} \mathcal{I}_{\boldsymbol{x}}
\end{align*}

\noindent Part 2 follows because of the distributivity of set operations and the fact that $\mathcal{I}_{\boldsymbol{x}} \cap \mathcal{I}_{\boldsymbol{\tilde{x}}} = \emptyset$.
 
\end{proof}

\begin{lem}[identified subset of $X^{i}_{\boldsymbol{x}^{\mathcal{S}}}$]\label{lemma:identification}
\begin{align*}
\mathcal{X}^{i}_{\boldsymbol{x}^{\mathcal{S}}} = \mathcal{I}_{\boldsymbol{x}^{\mathcal{S}}}
\end{align*}
\end{lem}

\begin{proof}
Slightly abusing notation, let us write $X^{i}_{\boldsymbol{x}} = X^{i}_{\boldsymbol{x}^{\mathcal{S}} \boldsymbol{x}^{\bar{\mathcal{S}}}}$, where $\boldsymbol{x}^{\mathcal{S}}$ specifies the elements of $\boldsymbol{x}$ indexed by $\mathcal{S}$ (and likewise for $\boldsymbol{x}^{\bar{\mathcal{S}}}$). We note that
\begin{align}
X^{i}_{\boldsymbol{x}^{\mathcal{S}}}(\omega) &\equiv \sum_{\boldsymbol{x}^{\bar{\mathcal{S}}}} I_{\boldsymbol{x}^{\bar{\mathcal{S}}}}(\omega)X^{i}_{\boldsymbol{x}^{\mathcal{S}}\boldsymbol{x}^{\bar{\mathcal{S}}}}(\omega) \nonumber \\
&= \sum_{\boldsymbol{x}^{\bar{\mathcal{S}}}} \left( \sum_{\boldsymbol{\tilde{x}}^{\mathcal{S}}} I_{\boldsymbol{\tilde{x}}^{\mathcal{S}}\boldsymbol{x}^{\bar{\mathcal{S}}}}(\omega) \right) X^{i}_{\boldsymbol{x}^{\mathcal{S}} \boldsymbol{x}^{\bar{\mathcal{S}}}}(\omega) \nonumber \\
&= \sum_{\boldsymbol{x}^{\bar{\mathcal{S}}}} \sum_{\boldsymbol{\tilde{x}}^{\mathcal{S}}} I_{\boldsymbol{\tilde{x}}^{\mathcal{S}}\boldsymbol{x}^{\bar{\mathcal{S}}}}(\omega) X^{i}_{\boldsymbol{x}^{\mathcal{S}} \boldsymbol{x}^{\bar{\mathcal{S}}}}(\omega) \label{eq:summation_identity}
\end{align}
\noindent The above summation shows that $X^{i}_{\boldsymbol{x}^{\mathcal{S}}}$ is identified at $\omega$ iff there exists $\boldsymbol{\check{x}}^{\mathcal{S}}$ and $\boldsymbol{\check{x}}^{\bar{\mathcal{S}}}$ such that $I_{\boldsymbol{\check{x}}^{\mathcal{S}}\boldsymbol{\check{x}}^{\bar{\mathcal{S}}}}(\omega) = 1$ and $X^{i}_{\boldsymbol{x}^{\mathcal{S}}\boldsymbol{x}^{\bar{\mathcal{S}}}}$ is identified. For $X^{i}_{\boldsymbol{x}^{\mathcal{S}}\boldsymbol{\check{x}}^{\bar{\mathcal{S}}}}$ to be identified at $\omega$, $I_{\boldsymbol{x}^{\mathcal{S}}\boldsymbol{\check{x}}^{\bar{\mathcal{S}}}}(\omega) = 1$. However, in the above double summation, exactly one of the $I_{\boldsymbol{\tilde{x}}^{\mathcal{S}}\boldsymbol{x}^{\bar{\mathcal{S}}}}(\omega)$ is equal to 1, with all others are equal to 0. Therefore, the condition for identification of $X^{i}_{\boldsymbol{x}^{\mathcal{S}}}$ at $\omega$ is reduced to the existence of $\boldsymbol{x}^{\bar{\mathcal{S}}}$ such that $I_{\boldsymbol{x}^{\mathcal{S}}\boldsymbol{x}^{\bar{\mathcal{S}}}}(\omega) = 1$, which is true iff $\sum_{\boldsymbol{x}^{\bar{\mathcal{S}}}} I_{\boldsymbol{x}^{\mathcal{S}} \boldsymbol{x}^{\bar{\mathcal{S}}}} = 1$, which is true iff $\omega \in \mathcal{I}_{\boldsymbol{x}^{\mathcal{S}}}$, as required. 
\end{proof}
                           
\end{document}